\documentclass{article}
\usepackage[utf8]{inputenc}
\usepackage{amsmath,amsthm, amssymb, amsfonts}
\usepackage{todonotes}
\usepackage{kantlipsum}
\usepackage{wrapfig}
\usepackage{paralist}
\usepackage{fullpage}
\usepackage{hyperref}
\usepackage{graphicx}
\usepackage{wrapfig}
\usepackage[]{algorithm2e}
\allowdisplaybreaks[1]
\usepackage{sistyle} % or siunitx
\newcommand\bSI[1]{{\small[\SI{}{#1}]}}

\makeatletter
\newlength\unitwdth
\newlength\numwdth
\newlength\tdima
\newcommand\SIdescr[2]{%
    \setlength\tdima{\linewidth}%
    \addtolength\tdima{\@totalleftmargin}%
    \addtolength\tdima{-\dimen\@curtab}%
    \addtolength\tdima{-\unitwdth}%
    \addtolength\tdima{-\numwdth}%
    \parbox[t]{\tdima}{%
        #1
        \leaders\hbox{$\m@th\mkern \@dotsep mu\hbox{\tiny.}\mkern \@dotsep mu$}%
        \hfill
        \ifhmode\strut\fi
        \makebox[0pt][l]{%
            \makebox[\unitwdth][l]{}%
            \makebox[\numwdth][r]{#2}}}}
\makeatother

\newcommand{\Hil}{\mathcal{H}}

\newcommand{\Acal}{\mathcal{A}}
\newcommand{\Bcal}{\mathcal{B}}

\newcommand{\Wcal}{\mathcal{W}}

\newcommand{\Z}{\mathbb{Z}}
\newcommand{\N}{\mathbb{N}}

\newcommand{\R}{\mathbb{R}}

\newcommand{\psitilde}{\tilde{\psi}}

\newcommand{\Fcal}{\mathcal{F}}

\newcommand{\fa}{, \quad \text{ for all }}

\newcommand{\bdWaveS}{\mathcal{W}} 
 
\newcommand{\Lp}[2]{L^{#1}{#2}}

\DeclareMathOperator{\suppp}{supp \,}

\def\diag{{\text{\rm diag}}}

\newtheorem{theorem}{Theorem}[section]
\newtheorem{remark}[theorem]{Remark}

\newtheorem{definition}[theorem]{Definition}

\newtheorem{corollary}[theorem]{Corollary}
\newtheorem{assumption}[theorem]{Assumption}

\title{Approximation properties of hybrid shearlet-wavelet frames for Sobolev spaces}

\author{Philipp Petersen\thanks{Mathematical Institute, University of Oxford, Andrew Wiles Building, Woodstock Road, Oxford, OX2 6GG, UK, e-mail: \texttt{Philipp.Petersen@maths.ox.ac.uk}} 
\and Mones Raslan\thanks{Institut f\"ur Mathematik, Technische Universit\"at Berlin, Straße des 17.~Juni 136, 10623 Berlin, Germany, e-mail: \texttt{raslan@math.tu-berlin.de}}
} 

\begin{document}
\maketitle
\begin{abstract}
In this paper, we study a newly developed hybrid shearlet-wavelet system on bounded domains which yields frames for $H^s(\Omega)$ for some $s\in \N$, $\Omega \subset \R^2$. We will derive approximation rates with respect to $H^s(\Omega)$ norms for functions whose derivatives admit smooth jumps along curves and demonstrate superior rates to those provided by pure wavelet systems. 
These improved approximation rates demonstrate the potential of the novel shearlet system for the discretization of partial differential equations. Therefore, we implement an adaptive shearlet-wavelet-based algorithm for the solution of an elliptic PDE and analyze its computational complexity and convergence properties.
\end{abstract}



\section{Introduction}

A well-established paradigm of applied harmonic analysis is the usage of representation systems to store, analyze and manipulate data. An especially famous representation system is provided by wavelets (see e.g. \cite{TenLectures}), which can be considered as a standard tool in signal and image processing as well as in the numerical analysis of partial differential equations. 
While being perfectly suited to approximate one-dimensional functions with point singularities, these systems admit one serious defect in 2D. In fact, two-dimensional wavelet systems yield significantly suboptimal approximation rates when dealing with functions that admit anisotropically structured singularities. Most importantly curvilinear singularities cannot be handled adequately. 
This fact is even more severe since such structures appear very frequently in natural data as, for instance, virtually every photograph admits at least one jump in color value along a line or a curve. 

As a remedy anisotropic directional systems were invented, starting with curvelets \cite{ candes1999curvelets,CurveletsIntro}. Later contourlets \cite{DoV2003Contourlets} and shearlets \cite{LLKW2007} were developed. 
All of these systems yield optimal approximation rates of functions that exhibit curvilinear singularities. 
From these systems, shearlets stand out due to their unique combination of desirable features which include a unified treatment of digital and continuous realm, compactly supported elements, fast implementations, and optimal approximation rates. 
Besides, the shearlet transform can be interpreted as a unitary representation of the so-called shearlet group. This offers a natural definition of associated smoothness spaces, customarily called shearlet coorbit spaces \cite{DKST2009, DahST2011coorbits}.
In addition, 2D wavelet systems, as well as shearlet systems, are both examples of the concept of wavelets with composite dilations \cite{GuoLLW2006theoryofcompdil}, making both systems very close conceptually such that we can think of shearlets as an extension of 1D wavelets to 2D which retains the optimal approximation properties for functions with jump singularities.

It is because of these properties that a very effective line of research deals with the identification of application areas of wavelets, where improvements can be made by switching to shearlet systems. Examples where wavelet results were improved using shearlets include denoising and inpainting \cite{shearLab, EasLC2009ShearDen}, edge detection \cite{king2015shearlet}, regularization of ill-posed problems \cite{BubLZBG2015roirec, HauS2013convexSegShear} and the reconstruction from Fourier samples \cite{Ma2016GenSamp}.

One of the strongholds of wavelet methods is their remarkable ability to yield efficient discretizations of PDEs. These discretizations lead to very effective numerical algorithms, which we will review below. These results are only possible because wavelet systems can be constructed to yield frames or Banach frames for Sobolev spaces. Additionally, wavelet systems can be built on bounded domains also incorporating boundary conditions. Since PDEs are usually considered on bounded domains, this property is crucial.

Following the idea laid out earlier, we plan to use shearlets---just like wavelets---for the discretization of PDEs. We build our results on a recently developed shearlet system on bounded domains \cite{GroKMP2015shearbddom}, which yields frames for $H^s(\Omega)$, while allowing for boundary conditions.

\subsection{Motivation: Adaptive frame methods}

A first important step in utilizing systems from harmonic analysis for the solution of partial differential equations was achieved by Cohen, Dahmen, and DeVore in \cite{cohen2001adaptive} who developed a wavelet-based solver for linear, elliptic PDEs guaranteeing optimal convergence and complexity provided that the solution of the PDE contains only point singularities. The results from \cite{cohen2001adaptive} were extended in \cite{Ste2003,DahFR2007,dahRWFS2007steepdesc} such that it becomes possible to work with general frames instead of wavelet bases. Furthermore \cite{DahFPRW2009nonlinFrame} broadened the results of \cite{DahFR2007} to nonlinear elliptic PDEs. 

Utilizing anisotropic frames for solving PDEs numerically is a relatively new topic, mainly due to the difficulty of constructing them on bounded domains.
Let us shortly review which conditions should be fulfilled such that the methods introduced in the aforementioned papers work efficiently. 

For linear problems, the aim is to obtain a well-conditioned discrete system of linear equations equivalent to the PDE by using a frame, which then can be solved by standard numerical methods. By now there exist two main approaches to end up with such a discrete system: the first one, introduced in \cite{Ste2003}, uses frames for the Sobolev space $H^s(\Omega),$ the solution space of the PDE. The second one uses the concept of Gelfand frames, where a frame for $L^2(\Omega)$ is used, which has the additional property that its synthesis operator is bounded as a map from a weighted sequence space into $H^s(\Omega)$ and the analysis operator with respect to its dual frame is a bounded map from $H^{s}(\Omega)$ into the sequence space.

Of course, in both cases, the frames should be constructed in such a way that boundary conditions on the PDE can be imposed. Furthermore, the involved frame should yield optimally sparse approximations to the solution in order to obtain optimal asymptotic convergence rates of the numerical algorithm to the solution with low computational work and high accuracy.

So far, some first attempts have been made in utilizing anisotropic frames for solving elliptic PDEs. In \cite{EGO14_1305,GO14_1334} optimal ridgelet-based solvers were developed for linear advection equations, whereas in \cite{ResAniStructures,DahKLSW2015} shearlet-based solvers for general advection equations were developed. Although these works constitute major successes in advancing anisotropic frame-based solvers, it was not possible to impose boundary conditions. 

To overcome these problems, a hybrid shearlet-wavelet frame was constructed in \cite{GroKMP2015shearbddom} which will be further reviewed in the following subsection.

\subsection{Anisotropic multiscale systems on bounded domains} 

Shearlet systems on $\R^2$ admit two central properties. First of all, it was demonstrated in \cite{GuoL2007, KLcmptShearSparse2011}, that they admit optimal approximation rates for so-called cartoon-like functions. Here, cartoon-like functions are functions which are piecewise $C^2$-functions on $\R^2$ with a $C^2$-discontinuity curve. 
In fact, using shearlets, the $L^2$-error of best $N$-term approximation for cartoon-like functions is decaying with a rate of $N^{-1}$ (up to logarithmic terms)---a considerable improvement compared to the approximation rate $N^{-1/2}$ provided by wavelets. We also mention a recent extension of analyzing the approximation rates of functions whose higher-order derivatives are cartoon-like functions \cite{Pet2014discontderiv}. Also in this scenario shearlets yield superior approximation rates over wavelets.

The second cornerstone of shearlet theory is provided by the fact that shearlet systems yield stable decompositions and reconstructions of functions in $L^2(\R^2)$, more specifically they form frames for $L^2(\R^2)$. A construction of compactly supported shearlets admitting this property was provided in \cite{KGLConstrCmptShear2012}.

While both these properties appear very beneficial especially given our long-term goal to use shearlets as a discretization tool for partial differential equations, we can observe that the standard construction is not yet fully satisfactory for that task. In fact, a significant obstacle that needs to be tackled is that the standard construction only constitutes a representation system for $L^2(\R^2)$, while a PDE is usually defined on a bounded domain $\Omega\subset\R^2.$ Therefore, it is necessary to introduce multiscale systems on bounded domains, which retain the frame and approximation properties of its $\R^2$ counterpart. 

After extensive research, this task is quite well understood for wavelets (see for example \cite{CohDauVial, CDDBoundaryWavelets,Coh2000, CanTU1999WaveElMeth}) but there are still open questions despite the isotropic character of wavelets. All the constructions of wavelets on bounded domains hinge upon the multiresolution structure of the systems in order to construct boundary adapted elements. Therefore it is clear that shearlet systems for bounded domains are even harder to construct since such a structure is missing for these systems. Furthermore, the anisotropic shape of the support can intersect the boundary to various degrees and at various angles requiring different boundary adaptation for each element. 

The first attempt in this direction has been made in \cite{GLShearletsonBoundedDomains2012}, where an $L^2$-frame with optimal approximation properties for cartoon-like functions was constructed. Unfortunately, the resulting system is not boundary adapted; hence it is not possible to characterize Sobolev spaces or impose boundary conditions on a PDE. Another attempt has been made in \cite{ResAniStructures,DahKLSW2015}, but the resulting systems neither constitute frames for $L^2(\Omega)$ nor characterize Sobolev spaces. 

A system that provides all the desired properties was finally developed in \cite{GroKMP2015shearbddom, DissPP}. It was demonstrated how to construct a hybrid frame for $H^s(\Omega)$, where $s \in \N_0= \N \cup \{0\}$ and $\Omega \subset \R^2$ by combining a shearlet frame with a wavelet frame on $\Omega$. Roughly speaking, the resulting frame consists in one part of all elements of some shearlet frame for $H^s(\R^2)$ with compact support (see for instance \cite{KGLConstrCmptShear2012, DahST2011coorbits} for a construction) the support of which is fully contained in $\Omega$. Additionally, boundary adapted wavelets in a small tubular region around the boundary are included in the system, so that it becomes possible to impose boundary conditions.

If a frame for $L^2(\Omega)$ is constructed in the just described way, then it can be shown that the system characterizes Sobolev space norms by weighted $\ell_2$ norms of the associated analysis coefficients. Moreover, it was numerically established in \cite{DissPP} that the system also yields a Gelfand frame so that it can be used for the numerical solution of PDEs using the Gelfand frame approach introduced in \cite{DahFR2007}.
Since only a few wavelets are used in the construction, it is still possible to show that the resulting system approximates an extended class of cartoon-like functions, defined on $\Omega$ optimally.

Additionally, by changing the construction slightly, the system also constitutes a frame for $H^s(\Omega)$, which is the set-up we will use in this work. In this case, the approximation properties of the system, at least with respect to $H^s(\Omega)$ norms have not yet been analyzed in \cite{GroKMP2015shearbddom}. 

\subsection{Our contribution}
In this paper, we provide further studies of a hybrid shearlet-wavelet frame for $H^s(\Omega)$. Our analysis is two-fold. First of all, we establish novel approximation rates for functions with curvilinear singularities within their derivatives with respect to pure shearlet frames for Sobolev spaces on $\R^2$ as well as on bounded domains $\Omega$ with respect to a hybrid shearlet-wavelet frame. We will establish improved rates over pure wavelet systems. Secondly, we implement an adaptive numerical algorithm for the solution of a model PDE. We will observe that the convergence of this algorithm matches our theoretical results and is of higher-order than that provided by wavelet discretizations. 
\subsection{Outline}

The paper is organized as follows. First of all, in Section 2 we present the main concepts needed throughout subsequent sections. A short revision of shearlets on $H^s(\R^2)$ for $s\in \N_0$ is included. Most importantly, it also contains a review of hybrid shearlet-wavelet frames on bounded domains for $H^s(\Omega)$ for some bounded domain $\Omega\subset\R^2$ which were introduced in \cite{GroKMP2015shearbddom,DissPP}. In Section 3 we prove approximation results with respect to the best $N$-term approximation in the $H^s(\Omega)$-norm. In particular, we will provide fast approximation rates for functions which have cartoon-like first- or higher-order derivatives and will also explain why these functions appear frequently as solutions of elliptic PDEs. In Section 4 we implement an adaptive algorithm based on the constructed Sobolev frame, solving an exemplary PDE. We will observe very fast convergence rates, which thereby highlight the potential of the new system.

\section{Preliminaries}

In this section, we first introduce shearlet and wavelet systems as well as basic notation.

\subsection{Basic notation}

For any Banach space $(X,\left\|\cdot \right\|_X)$ let $X'$ be its topological dual. We furthermore use the usual multiindex notation: for $ \textbf{a}= (a_1,...,a_d)\in \N_0^d$ let $|\textbf{a}|=a_1+...+a_d$ and $D^{\textbf{a}}=D_1^{a_1}\cdot\cdot\cdot D_d^{a_d}.$ On $\R^d$ the Euclidean scalar product shall be denoted by $\langle x, y\rangle$ and the induced norm by $|x|$ for $x,y\in\R^d.$ We also denote by $|x|$ the absolute value for some $x\in\mathbb{C}$. 
For a measurable subset $\Omega\subseteq \R^d$ denote by $L^p(\Omega),p\in [1,\infty]$ the usual Lebesgue spaces and for a countable index set $\Lambda$ and $p\in [1,\infty]$ by $\ell^p(\Lambda)$ the usual sequence spaces. The cardinality of some set $I$ shall be denoted by $\# I$ and the Lebesgue measure of some measurable set $\Omega\subset \R^2$ by $|\Omega|.$ The Fourier transform we use is given by $(\Fcal f)(\xi)=\hat{f}(\xi)=\int_{\R^d}f(x)e^{-2\pi i \langle x,\xi \rangle}dx$ for $f\in L^1(\R^d)\cap L^2(\R^d)$ which can be uniquely extended to $L^2(\R^d).$ We note that for $f\in L^2(\R^d)$ the Plancherel identity $\left\|f\right\|_2=\left\|\Fcal f\right\|_2$ holds and $\Fcal^{-1}$ exists. For any subset $\Omega\subseteq\R^d$ let $\partial \Omega$ be its boundary. 
For some Hilbert space $\Hil$ and some closed linear subspace $M$ of $\Hil,$ let $\mathbf{P}_M$ denote the orthogonal projection onto $M.$ For some normed space $X$, some $M\subseteq X$ and $x\in X$ let $d(x,M):=\inf \{\left\|x-y \right\|_X:y\in M \}$ denote the distance of $x$ to $M.$ Furthermore, for some $r>0$ let $B_r(x)$ denote the open ball with radius $r$ and center $x.$ If we have $A\leq C\cdot B$ for two quantities $A,B$ and some constant $C>0$ we write $A\lesssim B$ and $A\gtrsim B$ if $B\lesssim A.$ Furthermore $A\sim B$ shall be written if $A\lesssim B$ and $B\lesssim A$ hold. Furthermore, for a real number $x$ we denote by $\lfloor x \rfloor$ the largest integer less than $x$.
 
\subsection{Frames}

In this subsection, we will introduce \emph{frames} for Hilbert spaces. 
\begin{definition}[\cite{DSFrames1952, Chr}]
A countable subset $\left(\varphi_\lambda \right)_{\lambda\in\Lambda}$ of some Hilbert space $\Hil$ is called a \emph{frame} for $\Hil,$ if there exist $0<\Acal\leq \Bcal<\infty$, such that \begin{align} \label{eq:frameprop}
    \Acal \left\| f \right\|_{\Hil}^2\leq \sum_{\lambda\in\Lambda} |\langle f,\varphi_\lambda\rangle_{\Hil}|^2 \leq \Bcal \left\| f \right\|_{\Hil}^2 \fa f\in \Hil.
\end{align}

\end{definition}

For every frame $\Phi=(\varphi_\lambda)_{\lambda\in\Lambda}$, we can define the linear and bounded \emph{analysis operator}  ${T_{\Phi}:\Hil\rightarrow\ell^2(\Lambda),}$ \\  ${f\mapsto (\langle f,\varphi_\lambda\rangle_{\Hil})_{\lambda\in\Lambda}}$ as well as its adjoint, the \emph{synthesis operator,} given by  ${T_{\Phi}^{*}:\ell^2(\Lambda)\rightarrow\Hil,}$ \\ ${(c_\lambda)_{\lambda\in\Lambda}\mapsto \sum_{\lambda\in\Lambda} c_\lambda \varphi_\lambda.}$ Then, due to the frame property \eqref{eq:frameprop}, it is easy to see   that the \emph{frame operator} $S_\Phi=T^*_\Phi T_\Phi$ is a bounded and boundedly invertible operator from $\Hil$ onto itself. Therefore, we can define the \emph{canonical dual frame} $(\varphi_\lambda^{\mathrm{dual}})_{\lambda\in\Lambda}:=(S_\Phi^{-1}\varphi_\lambda)_{\lambda\in\Lambda},$ which is a frame itself. Afterwards, one can obtain the reconstruction formula 
\[
f=\sum_{\lambda\in\Lambda}\langle f,\varphi_\lambda^{\mathrm{dual}} \rangle_{\Hil} \varphi_{\lambda}=\sum_{\lambda\in\Lambda}\langle f,\varphi_\lambda \rangle_{\Hil} \varphi_{\lambda}^{\mathrm{dual}},
\] which holds for every $f\in\Hil.$ 

Let $\Phi=(\varphi_\lambda)_{\lambda\in\Lambda}$ be a frame for a Hilbert space $\Hil$. In view of real-world applications, it is reasonable to assume that we cannot store all frame coefficients or all frame elements. Hence, for a given ${f=\sum_{\lambda\in\Lambda} c_\lambda\varphi_\lambda\in\Hil}$ we would like to approximate $f$ in an optimal way by only using linear combinations of $N$ frame elements. To make this formal, following \cite{Dev1998}, we define the \emph{space of nonlinear approximation} by
\begin{align*}
\Sigma_N:=\bigcup_{\# M\leq N}\text{span}\{\varphi_\lambda: \lambda\in M\}.
\end{align*}
Then
\[\sigma_{N,\Phi,\Hil}(f):= \inf_{f_N\in\Sigma_N}\left\|f-f_N \right\|_{\Hil} \] is called \emph{the error of best N-term approximation of $f$ with respect to $\Phi$.} There exists a very important connection between the error of best $N$-term approximation and \emph{weak} $\ell^p$-\emph{spaces} which are defined by
\begin{align*}
    \ell^{p}_w:=\{(c_j)_{j\in\mathbb{N}}: \sup_{n\in\mathbb{N}} n^{1/p}c_n^*<\infty\},
\end{align*} where $(c_n^*)_{n\in\mathbb{N}}$ is a rearrangement of $\mathbf{c}=(c_\lambda)_{\lambda\in\Lambda}\in\ell^2(\Lambda)$ such that $|c^*_{n+1}|\geq |c^*_{n}|$ for all $n\geq 1.$ It was shown in \cite{Dev1998} that for $\textbf{c} \in \ell^2$ we have that 
\[
\sigma_{N,\Phi,\Hil}(\mathbf{c})\lesssim N^{-s} \iff \textbf{c}\in\ell^{p}_w
\]
for $p=(\frac{1}{2}+s)^{-1}$, where $\Hil=\ell^2(\N)$ and the $N$-term approximation is taken with respect to the canonical basis of $\ell^2(\N)$. If $(\varphi_n)_{n\in \N}$ is a frame for some Hilbert space $\Hil$ one can still get 
\begin{align}\label{BestN-Term-weak}
 (\langle f,\varphi_n\rangle_\Hil )_{n\in \N}\in\ell^{p}_w \implies \sigma_{N,\Phi,\Hil}(f)\lesssim N^{-s},
\end{align}
where the $N$-term approximation is taken with respect to the canonical dual frame.
We add that $\ell^{p} \hookrightarrow\ell^{p}_w \hookrightarrow \ell^{p+\epsilon}$ for $0<p<2$ and $\epsilon>0$.

\subsection{Function spaces}
Since we aim to derive approximation properties of shearlet frames for \emph{Sobolev spaces} on $\R^2$ as well as on bounded domains, we briefly recall their definition.
\begin{definition}
Let $\Omega \subseteq \R^2$ be an open domain and $s\in\N.$ Then the \emph{Sobolev space $W^{s,2}(\Omega)$ of order $s$} is given by
\[
W^{s,2}(\Omega):= \{f\in L^2(\Omega): D^{\mathbf{a}}f\in L^{2}(\Omega)~\text{for all}~|\mathbf{a}|\leq s\}.
\] 
These spaces, equipped with the inner product $\langle f,g \rangle_{W^{s,2}(\Omega)}:= \sum_{|\mathbf{a}|\leq s} \langle D^{\mathbf{a}}f , D^{\mathbf{a}}g \rangle_{L^2(\Omega)},$ are Hilbert spaces.
\end{definition}

If $\Omega=\R^2,$ then we can characterize $W^{s,2}(\R^2)$ by using \emph{Bessel potential spaces}: 
\begin{definition}
Let $s\in\mathbb{N}.$ Then the \emph{Bessel potential space of order $s$} is given by 
\[H^{s}(\R^2):=\{f\in L^{2}(\R^2):\mathcal{F}^{-1}[(1+|\cdot|^2)^{\frac{s}{2}}\mathcal{F}f]\in L^{2}(\R^2)\}.\]
These spaces, equipped with the inner product $\langle f,g \rangle_{H^{s}(\R^2)}:= \langle (1+|\cdot|^2)^{\frac{s}{2}}\hat{f},(1+|\cdot|^2)^{\frac{s}{2}}\hat{g} \rangle_{L^2(\R^2)}$, are Hilbert spaces. 
\end{definition}

In fact, it holds that $\left\|\cdot \right\|_{W^{s,2}(\R^2)} \sim \left\|\cdot \right\|_{H^{s}(\R^2)}$, which is why in the coming sections we will sometimes work on $H^s(\R^2)$ instead of $W^{s,2}(\R^2).$ If $\Omega\subsetneq \R^2,$ we will always work with $W^{s,2}(\Omega),$ but in order to simplify notation we will write $H^{s}(\Omega):=W^{s,2}(\Omega).$  
We furthermore define ${H_0^1(\Omega):=\{u \in H^1(\Omega):u=0 \text{ on } \partial \Omega\}}$ and $H^{-1}(\Omega):=(H_0^1(\Omega))'.$ On top of that, by $B_{s,p}^t$ we denote the usual Besov spaces.

\subsection{Wavelet systems}

Let $\Omega \subset \R^2$ be a bounded open domain. In the sequel, we will work with wavelet systems on bounded domains. Later on we will assume a number of properties that the wavelet systems should have. For now we only stipulate that they should be indexed in a specific way. 
\begin{definition}
A set $\bdWaveS \subset \Lp{2}{(\Omega)}$ is called a \emph{boundary wavelet system} for some $J_0 \in \Z,$ if it can be written as
\begin{align*}
\bdWaveS = \{ \omega_{J_0, m, 0}: m \in K_{J_0}\} \cup \{ \omega_{j, m, \upsilon}: j\geq J_0, m \in K_{j}, \upsilon = {1,2,3}\},
\end{align*}
where $K_j\subset \Omega$, $\# K_j \sim 2^{2 j}$. We denote the index set by
$$
\Theta: =\{(J_0 ,m, 0): m\in K_{J_0}\} \cup \{(j,m, \upsilon): j\geq J_0, m\in K_j ,\upsilon = {1,2,3}\}.
$$
\end{definition}
Explicit constructions of boundary wavelet systems that yield biorthogonal bases for $L^2(\Omega)$ can be found, for example, in \cite{CohDauVial, DahS1998,DahKU1996stokes,CanTU1999WaveElMeth,bitt2006biortho,primbs2006stabile}.

\subsection{Shearlet systems on \texorpdfstring{$\R^2$}{R2}}\label{sec:shearletsonR^2}

Shearlet systems were introduced in \cite{GKL2006, GuoL2007} with the intent to improve on suboptimal approximation rates of wavelet systems for natural images.

The key towards faster approximation rates for functions displaying curvilinear singularities is to replace the isotropic scaling of wavelets by an anisotropic variant paired with an operation to adjust the orientation of the elements of the system. Towards such a construction, we define for $j,k \in \Z$ the \emph{parabolic scaling and shearing matrices} by
\begin{align*}
    A_j = \left(\begin{array}{ll}
         2^j& 0  \\
         0 & 2^{\frac{j}{2}}\\ 
    \end{array}\right), \text{ and } 
    S_k = \left(\begin{array}{ll}
         1& k  \\
         0 & 1\\ 
    \end{array}\right).
\end{align*}

Using the matrices above, we present the definition of a \emph{cone-adapted shearlet system}:

\begin{definition}[\cite{KLcmptShearSparse2011, GKL2006}]\label{def:ShearletSystem}
Let $\phi, \psi \in L^2(\R^2)$, $c= [c_1,c_2]^T \in \R^2$ with $c_1,c_2>0$. Then the \emph{cone-adapted shearlet
system} is defined by

\begin{equation*}
\mathcal{SH}(\phi, \psi, \tilde{\psi}, c) := \Phi(\phi, c_1) \cup \Psi(\psi, c) \cup \tilde{\Psi}(\tilde{\psi}, c),
\end{equation*}
where
\begin{align*}
\Phi(\phi, c_1) :=& \left \{ \psi_{0,0,m,0} = \phi(\cdot - c_1 m) :m\in \Z^2 \right\},\\
\Psi(\psi, c) :=& \left\{ \psi_{j,k,m,1} = 2^{\frac{3j}{4}}\psi(S_k A_{j}\cdot -  M_c m ): j\in \N_0, |k| \leq   2^{\left\lceil\frac{j}{2}\right\rceil}, m\in \Z^2 \right\},\\
\tilde{\Psi}(\tilde{\psi}, c) :=&\left\{ \psi_{j,k,m,-1} = 2^{\frac{3j}{4}}\tilde{\psi}(S_k^T \tilde{A}_{j}\cdot -  M_{\tilde{c}} m ): j\in \N_0, |k| \leq
2^{\left\lceil\frac{j}{2}\right\rceil}, m\in \Z^2 \right\},
\end{align*}

with $\tilde{\psi}(x_1,x_2) = \psi(x_2,x_1)$, $M_c:= \diag( c_1, c_2 )$, $M_{\tilde{c}} = \diag( c_2 , c_1 )$, $\tilde{A}_{j} = \diag(2^{j/2},2^{j})$.
\end{definition}

For cone-adapted shearlet systems we will employ the index set 
\begin{align*}
\Lambda := \{(j,k,m,\iota): |\iota| j \geq j\geq 0, |k|\leq |\iota| 2^{\frac{j}{2}}, m\in \Z^2, \iota \in \{1,0,-1\} \}.
\end{align*}
Then we can write the cone-adapted shearlet system as $(\psi_{j,k,m,\iota})_{(j,k,m,\iota)\in \Lambda}$.

It was shown in \cite{KGLConstrCmptShear2012, DahST2011coorbits} that compactly supported shearlets can be constructed to yield frames for $L^2(\R^2)$. A further sufficient condition was presented in \cite{LimShearlets}. An easily fulfillable condition, under which reweighted shearlet systems $(2^{-js}\psi_{j,k,m,\iota})_{(j,k,m,\iota)\in\Lambda}$ constitute frames for $H^s(\R^2),s\geq 0,$ was given in \cite{GroKMP2015shearbddom}.

For $\nu>0$ let $STAR^2(\nu)$ denote the set of all star-shaped subsets of $\R^2$ with $C^2$-boundary and curvature bounded by $\nu.$ Then we can define the set of all cartoon-like functions in the following way:

\begin{definition}
For $\nu>0$, let $\mathcal{E}^2(\nu)$ be the set of all functions $f\in L^2(\R^2),$ for which there exist some $B\in STAR^2(\nu)$ and $f_i\in C^2(\R^2)$ with compact support in $(0,1)^2$ as well as $\left\|f_i \right\|_{C^2(\R^2)}\leq 1$ for $i=1,2$ such that 
\[f=  f_1+ \mathcal{X}_B f_2.\]
We call $\mathcal{E}^2(\nu)$ the set of \emph{cartoon-like functions}.
\end{definition}

It was established in \cite{DCartoonLikeImages2001} that for an arbitrary dictionary $\Phi=(\varphi_i)_{i\in I}$ for $L^2(\R^2)$ the optimal achievable best $N$-term approximation rate for the class of cartoon-like functions is given by 
$$
\sigma_{N,\Phi,L^2(\R^2)}(f) = O(N^{-1}),
$$
provided that only polynomial depth search is used to compute the approximation. Let us now describe the approximation rate which shearlets achieve.
First of all, for a shearlet system $(\psi_{j,k,m,\iota})_{(j,k,m,\iota)\in \Lambda}$ and $f\in L^2(\R^2)$ 
we denote by $(c_n^{*}(f))_{n\in \N}$ the non-increasing rearrangement of $(|\langle f, \psi_{j,k,m,\iota}\rangle_{L^2(\R^2)} |^2)_{(j,k,m,\iota) \in \Lambda}$. 
Under mild assumptions on the generator functions $\phi$ and $\psi$ it was shown in \cite{GuoL2007, KLcmptShearSparse2011} that
\begin{align}\label{eq:optApprox123}
   \sum_{n\geq N} c_n^*(f) \lesssim N^{-2}\log(N)^3.
\end{align}
If $(\psi_{j,k,m,\iota})_{(j,k,m,\iota)\in \Lambda}$ constitutes a frame, then \eqref{eq:optApprox123} yields that the best $N$-term approximation rate with respect to any dual frame of the shearlet system $\Phi^{\mathrm{dual}}$ obeys
\begin{align*}
   \sigma_{N,\Phi^{\mathrm{dual}},L^2(\R^2)}(f) = O(N^{-1}\log(N)^\frac{3}{2}) \text{ for all } f \in \mathcal{E}^2(\nu),
\end{align*}
and, for the sequence $(c_n^*(f))_{n\in\N}$, we have $c_n^*(f)\lesssim n^{-3}\log(n)^{3}.$ One aim of this paper is the generalization of these results to functions in $H^s(\R^2)$ whose higher-order derivatives are cartoon-like.

\subsection{Hybrid shearlet-wavelet systems on bounded domains}

In \cite{GroKMP2015shearbddom}, a hybrid shearlet-wavelet system on a bounded domain $\Omega\subset \R^2$ was constructed, resulting from a combination of parts of a wavelet frame on a bounded domain and parts of a shearlet system on $\R^2$. In order to still obtain good approximation properties it is essential not to include all elements of the wavelet system in the hybrid system. In particular, the elements which will be included should be located only in a thin strip $\Gamma_{\gamma(j)}$ close to the boundary of $\Omega,$ where for $r>0$ and some fixed $q_{\mathrm{sh}}>0$, which will be specified later, 
\[
\Gamma_r= \{x \in \Omega : d(x,\partial \Omega) < q_{\mathrm{sh}} 2^{-r}\},
\] 
$\gamma(j)$ depends on the scale $j$ of the wavelets, and $\Gamma_{\gamma(j)}$ shrinks for increasing scales.  Before we continue with the defintion of a hybrid shearlet-wavelet system and the analysis of the frame property, we need to introduce the following assumptions imposed on the underlying wavelet and shearlet systems.

\begin{assumption}\label{Ass:ShWvlet} \cite{GroKMP2015shearbddom} Let $s\in\N_0$, $\mathcal{W}$ be a boundary wavelet system and $\mathcal{SH}(\phi,\psi,\psitilde,c)$ be a shearlet system. Then, we assume the following properties of the boundary wavelet system:
\begin{itemize}

\item[(W1)] $(2^{-js}\omega_{j,m,\upsilon})_{(j,m,\upsilon) \in \Theta}$ is a frame for $H^s(\Omega)$ and there exists a dual frame \\
$\mathcal{W}^{\mathrm{dual}}_{H^s}=(2^{-js}\omega_{j,m,\upsilon})^{\mathrm{dual}}_{(j,m,\upsilon) \in \Theta}$ and for all $(j,m,\upsilon)\in \Theta$ with ${\partial \Omega \cap \suppp (2^{-js}\omega_{j,m, \upsilon})^{\mathrm{dual}} = \emptyset}$ it holds that 
\begin{align}\label{eq:RegWave1}
|\widehat{(2^{-js}\omega_{j,m, \upsilon}})^{\mathrm{dual}}(\xi)| \lesssim 2^{-js}\cdot 2^{-j}\frac{\min\{1,|2^{-j}\xi_i|^{\alpha_{\mathrm{w}}}\}}{\max\{1,|2^{-j}\xi_1|^{\beta_{\mathrm{w}}}\}\max\{1,|2^{-j}\xi_2|^{\beta_{\mathrm{w}}}\}}, 
\end{align}
for at least one $i\in \{1,2\},$ some $\alpha_{\mathrm{w}},\beta_{\mathrm{w}}>0$ and all $\xi \in \R^2$.
Here the Fourier transform is to be understood on $\Lp{2}{(\R^2)}$ after extension by $0$. Furthermore, we assume that the elements of $\bdWaveS^{\mathrm{dual}}_{H^s}$ have compact support and let 
$$
q_{\mathrm{w}}^{(0)} := \inf \left\{ q>0 : \suppp (2^{-js}\omega_{j,m, \upsilon})^{\mathrm{dual}}\subset B_{2^{-j}q}(m) \fa (j,m,\upsilon) \in \Theta\right\}>0.$$

\item[(W2)] The elements of $\bdWaveS$ have compact support and 
$$
q_{\mathrm{w}}^{(1)} := \inf \left\{q>0: \suppp \omega_{j,m, \upsilon}\subset B_{2^{-j}q}(m) \fa (j,m,\upsilon) \in \Theta\right\} >0.
$$
Moreover, we have that 
$$
|m - m'| \geq 2^{-j}q_{\mathrm{w}}^{(2)}\fa j\geq J_0 \text{ and }m,m' \in K_j,m \neq m',
$$
for some $q_{\mathrm{w}}^{(2)}>0.$
\end{itemize}
Furthermore, we assume the following properties of $\mathcal{SH}(\phi,\psi,\psitilde,c)$:
\begin{itemize}

\item[(S1)] $(2^{-js}\psi_{j,k,m,\iota})_{(j,k,m,\iota) \in \Lambda}$ is a frame for $H^s(\R^2)$ with dual frame $(2^{-js}\psi_{j,k,m,\iota})^{\mathrm{dual}}_{(j,k,m,\iota) \in \Lambda}$.

\item[(S2)] For some $C_1,C_2 >0$ the decay conditions
\begin{align}\label{eq:RegShear1}
|\widehat{\psi}(\xi_1,\xi_2)| \leq C_1 \frac{\min\{1,|\xi_1|^{\alpha_{\mathrm{sh}}}\}}{\max\{1,|\xi_1|^{\beta_{\mathrm{sh}}}\}\max\{1,|\xi_2|^{\beta_{\mathrm{sh}}}\}}
\end{align}
and
\begin{align}\label{eq:RegShear2}
|\widehat{\psitilde}(\xi_1,\xi_2)| \leq C_2 \frac{\min\{1,|\xi_2|^{\alpha_{\mathrm{sh}}}\}}{\max\{1,|\xi_1|^{\beta_{\mathrm{sh}}}\}\max\{1,|\xi_2|^{\beta_{\mathrm{sh}}}\}}
\end{align}
are obeyed for all $(\xi_1, \xi_2) \in \R^2$ and some $\alpha_{\mathrm{sh}},\beta_{\mathrm{sh}}>0$.

\item[(S3)] For all $(j,k,0,\iota)\in \Lambda$ and some $q_{\mathrm{sh}}>0$ we have that
$$
\suppp \psi_{j,k,0,\iota} \subset B_{2^{-j/2}q_{\mathrm{sh}}/2}(0).
$$
\end{itemize}
\end{assumption}
Essentially, Assumption \ref{Ass:ShWvlet} requires a wavelet frame and a shearlet frame for $H^s(\Omega)$, $H^s(\R^2)$ respectively. Moreover, the shearlet and wavelet frames are required to be built from compactly supported elements. For the wavelet system it is additionally required that the associated dual frame is compactly supported as well. For wavelet systems forming orthonormal bases this property is automatically satisfied. Finally, three regularity assumptions are made in \eqref{eq:RegWave1} \eqref{eq:RegShear1}, and \eqref{eq:RegShear2}. In combination with the assumptions on the supports, these regularity assumptions are necessary to control the correlation between wavelets elements and shearlets elements in the following definition of a hybrid shearlet-wavelet system. The construction of the hybrid shearlet-wavelet system functions by including wavelet elements supported close to the boundary of the domain and shearlet elements that are supported inside the domain. 

\begin{definition}\label{def:variableWidthShearletSystem}
Let $s\in \N_0$, $\mathcal{SH}(\phi,\psi,\widetilde{\psi},c) = (\psi_{j,k,m,\iota})_{(j,k,m,\iota) \in \Lambda}$ be a shearlet system fulfilling (S3), let $\tau>0$ and $t > 0$. Further, let $\mathcal{W}$ be a boundary
wavelet system fulfilling (W1) and (W2) and set
\begin{align*}
\mathcal{W}_{t,\tau}:= \left\{\omega_{j,m,\upsilon} \in \mathcal{W}: (j,m,\upsilon) \in \Theta_{t,\tau} \right\},
\end{align*}
where
\begin{align*}
    \Theta_{t,\tau} := \left\{(j,m,\upsilon) \in \Theta: B_{2^{-j}(q_{\mathrm{w}}^{(0)} + q_{\mathrm{w}}^{(1)})}(m) \cap \Gamma_{\tau(j-t)} \neq \emptyset\right\}.
\end{align*}
In addition, let
\[
    \Lambda_0: = \left\{ (j,k,m,\iota)\in \Lambda :  \suppp \psi_{j,k,m,\iota} \subset \Omega\right\}.
\]
Then, the \emph{hybrid shearlet-wavelet system with offsets $t$ and $\tau$}
is defined as
\begin{align*}\mathcal{HSW}_{t,\tau}^s(\Wcal; \phi, \psi, \psitilde, c):= \left\{ \psi_{j,k,m,\iota}:  (j,k,m,\iota)\in \Lambda_0\right\} \cup \mathcal{W}_{t,\tau}.
\end{align*}
\end{definition}

The parameters $\tau$ and $t$ in Definition \ref{def:variableWidthShearletSystem} control 
the size of the boundary strip $\Gamma_{\tau(j-t)}$ and thereby determine which wavelets from the underlying wavelet frame are included in the hybrid shearlet-wavelet system. Intuitively a large $\tau$ and a small $t$ imply an asymptotically quickly shrinking boundary strip for $j \to \infty$. On the other hand, decreasing $\tau$ or increasing $t$ will increase the width of the strip. 

The frame property for $H^s(\Omega)$ of a hybrid shearlet-wavelet system on \emph{minimally smooth} domains $\Omega \subset \R^2$ (see \cite{Stein1970SingIng}) has been analyzed in \cite{GroKMP2015shearbddom}:

\begin{theorem}[\cite{GroKMP2015shearbddom}] \label{thm:FrameProperty}
Let $\mathcal{SH}(\phi, \psi, \psitilde, c)$ and $\mathcal{W}$ satisfy (S1), (S2), (S3) and (W1), (W2), respectively, with $s\in\N_0,$ $\alpha_{\mathrm{w}} > 1$, $\alpha_{\mathrm{sh}} >0,$ $\beta_{\mathrm{w}} \geq s,$ $\beta_{\mathrm{sh}}>2+2s+2\alpha_{\mathrm{w}}$, $\tau > 0$ and $\epsilon >0$ such that ${((1-\epsilon)/\tau - 2)\alpha_{\mathrm{w}} >\frac{5}{2}}$.
Then there exists some $T>0$ such
that, for any $t\geq T$, then we have that the hybrid shearlet-wavelet system ${(\varphi_n)_{n\in \N} = \mathcal{HSW}_{t,\tau}^s(\bdWaveS; \phi, \psi, \psitilde, c)}$ satisfies
\begin{align*}
\|f\|_{H^s(\Omega)}^2  \sim \sum_{n\in \N} \left|\left \langle f, 2^{-j_n s}\varphi_{n} \right \rangle_{H^s(\Omega)} \right|^2 \fa f\in H^s(\Omega). 
\end{align*}
\end{theorem}

As already outlined in Subsection \ref{sec:shearletsonR^2}, the second core property of shearlets, besides the frame property, is the optimal approximation rate of shearlets for cartoon-like functions, which are defined on $\Omega$ in the following way:

\begin{definition}\label{def:CartoonOnOmega}
Let $\nu>0$, and $\Omega \subset \R^2$ be a domain, $B \in STAR^2(\nu)$, and $f_i \in C^2(\R^2)$, $\|f_i\|_{C^2(\R^2)}<\infty$ for $i= 1,2$.
If $\#(\partial B \cap \partial \Omega) \leq M$ for some $M\in \N$ and $\partial \Omega$ and $\partial B$ only intersect transversely, then $\mathbf{P}_{\Omega}(f_1 + \chi_B f_2)$ is a \emph{cartoon-like function on $\Omega$}. We denote the \emph{set of cartoon-like functions on $\Omega$} by $\mathcal{E}^2(\nu, \Omega)$.
\end{definition}

It was shown in \cite{GroKMP2015shearbddom, DissPP} that for all functions $f \in H^l(\Omega)$ such that for all $|\mathbf{a}| = l$, $D^{\mathbf{a}}f \in \mathcal{E}^2(\nu, \Omega)$ one obtains
\begin{align}\label{eq:ApproxRateab}
\sigma_{N,\Phi^{\mathrm{dual}},L^2(\Omega)}(f)= O(N^{-\frac{l+2}{2}}\log(N)^\frac{3}{2}),
\end{align}
with respect to any dual frame of the hybrid shearlet-wavelet system and the $L^2(\Omega)$-norm.
For $l = 0$ and up to the $\log$ factor, this is the optimal approximation rate.

In Section \ref{sec:ApproxRates}, we will prove a similar approximation rate to \eqref{eq:ApproxRateab}. There we measure the best $N$-term approximation error with respect to the hybrid shearlet-wavelet system and the $H^l(\Omega)$-norm instead of the $L^2(\Omega)$-norm.

\section{Approximation rates}\label{sec:ApproxRates}

The two essential properties of shearlet systems on $L^2(\R^2)$ and $L^2(\Omega)$, for a domain $\Omega \subset \R^2$, are the frame property as well as the property to obtain optimally sparse approximations of functions that exhibit anisotropic structures, in our case cartoon-like functions. We recalled these properties in Subsection \ref{sec:shearletsonR^2}. As we have seen already, reweighted shearlet systems will constitute frames for the Sobolev spaces $H^s(\R^2)$ or $H^s(\Omega)$, for $s\in \N_0$.

Certainly, cartoon-like functions with jump singularities do not belong to any Sobolev space $H^s(\Omega)$ for $s\geq\frac{1}{2}$. Hence, we need to introduce a suitable and relevant generalization of the cartoon-model to Sobolev spaces. Given that we are interested in capturing curve-like singularities we will analyze \emph{functions with cartoon-like derivatives}, i.e., functions $f$ such that for all multi-indices $|\mathbf{b}|\leq s$ we have that $D^\mathbf{b}f \in \mathcal{E}^2(\nu, \Omega)$. Of course this definition includes functions with smooth derivatives.

In the upcoming Subsection \ref{eq:regOfSols} we will demonstrate that this class of functions with cartoon-like derivatives does appear in practice, particularly as solutions of elliptic PDEs.

Thereafter we will establish approximation rates admitted by reweighted shearlet systems for these functions in various Sobolev norms. We will focus on the analysis of the case of bounded domains since this is more involved. Nonetheless, all approximation results hold for $H^s(\R^2)$ with minor changes.
Similarly to the approximation rates on $L^2(\R^2)$ we will observe that for this function class our shearlet system admits considerably faster approximation rates than wavelets in Subsection \ref{sec:comp2Wave}.

\subsection{Regularity of solutions of elliptic PDEs}\label{eq:regOfSols}

In this subsection we will recall some literature describing the regularity of the derivatives of the solution of an elliptic PDE. We will ultimately see that the solution of an elliptic PDE with cartoon-like right-hand side will have cartoon-like first- or higher-order derivatives.

We consider PDEs of the type:
\begin{align}\label{eq:ellipticeq}
      Lu=-\sum_{1\leq i,j\leq2}D_i(a_{ij} D_ju)+\sum_{1\leq i\leq2} b_i D_iu+c u=f,
\end{align}
on a bounded domain $\Omega\subset \R^2,$ where $a_{ij},b_i,c$ are bounded and measurable and $L$ fulfills the \emph{uniform ellipticity condition} 
\begin{align} \label{eq:ellcond}
\lambda |x|^2\leq \sum_{1\leq i,j \leq 2} a_{ij}(x)x_ix_j \fa x\in \Omega \text{ and some $\lambda>0.$}      
\end{align} 
In particular, the special case $L=-\Delta$ is included in the setting above. Furthermore, we are interested in the case that $f$ is a cartoon-like function on $\Omega$ or $f = D_1 g_1 + D_2 g_2\in H^{-1}(\Omega)$, where $g_1,g_2$ are cartoon-like functions. 
To have any chance of good approximation rates by shearlets, we need the solution $u$ to be smooth away from curvilinear singularities. A first naive approach in this direction would be to examine standard Sobolev estimates in order to analyze the global regularity of the solution $u$. A standard elliptic result, see e.g. \cite{DobroElliptic} yields that for the homogeneous Dirichlet problem (i.e., \eqref{eq:ellipticeq} on $\Omega$ and $u=0$ on $\partial \Omega$), provided that $f\in H^m(\Omega),\partial \Omega\in C^{m+2}$ and $a_{ij},b_i,c$ are sufficiently smooth, the solution $u\in H^{m+2}(\Omega)$. However, due to the jump curve of the right-hand side $f$, in general we have that $f\notin H^m(\Omega)$ for any $m\geq 1$. Hence, we can only expect $u\in H^2(\Omega).$
To improve upon such estimates it is clear, that we need local estimates to show that the solution $u$ is piecewise smooth.
As a first approach in this direction, in the case where $f$ is cartoon-like, we can employ classical domain decomposition and Schauder estimates, see \cite{GilTruElippticPDE2001}, to observe that $(D^{\alpha}u) \in C^{2}(B) \cup C^{2}(\Omega \setminus \bar{B})$. It is, however, still possible that $D^\alpha u$ explodes near $\partial B$. This would render an analysis of shearlet approximation rates more complicated.
If in addition $(D^{\alpha}u)_{|B} \in C^{2}(\bar{B})$ and $(D^{\alpha}u)_{|B^c} \in C^{2}(\bar{B^c})$, then it would immediately follow by the well-known Whitney extension theorem \cite{Whitney1992} that $D^{\alpha}u$ is cartoon-like on $\Omega$. In this scenario the approximation rates that will be established in the following subsection can be used.

Such an analysis of the behavior of $D^\alpha u$ close to the discontinuity of $f$ can be found in \cite{LiNirenberg2003compositematerial}. We briefly describe their setup. Let us now consider

\begin{align*}
    Lu=-\sum_{1\leq i, j \leq2} D_i(a_{ij}D_j u ) = f \quad \text{ with } f= \sum_{i} h_i  + D_i g_i,
\end{align*}
where $a$ fulfills the uniform ellipticity condition \eqref{eq:ellcond}.
For the special case that $h_1$ is cartoon like, $h_2, g_1, g_2= 0$ and $a = \begin{pmatrix} 1 & 0\\ 0 &1 \end{pmatrix}$ we have a standard Poisson equation with cartoon-like right hand side. 

Under the assumptions that we can decompose $\Omega$ as $\Omega = \bigcup_{m = 1}^N \Omega_m$, where $\Omega_m \subset \Omega,\partial \Omega_m \in C^\infty$ and we further assume that ${a_{i,j}}_{|\Omega_m}, {h_i}_{|\Omega_m}, {g_i}_{|\Omega_m} \in C^\infty(\bar{\Omega}_m)$ for all $m\leq N,$ we have that $u \in C^{\infty}(\bar{\Omega}_m\cap \Omega)$ \cite[Proposition 1.4]{LiNirenberg2003compositematerial}.

This result implies, that if we stipulate some additional conditions on the boundary behavior of $u$, we can conclude, that if $f$ is cartoon-like and $u$ is a solution of 
\begin{align}\label{eq:PDE12345}
    L u = f,
\end{align}
then the second derivatives of $u$ are cartoon-like. Furthermore, using the same argument, if \\ ${f = D_1 g_1 + D_2 g_2 \in H^{-1}(\Omega)}$, with $g_1,g_2$ being cartoon-like, then the solution $u$ of \eqref{eq:PDE12345} has derivatives, which are cartoon-like. 

\subsection{Approximation rates for functions with cartoon-like derivatives}

We provide approximation rates for two cases. First of all, we examine the approximation rates of reweighted hybrid shearlet-wavelet systems in the $H^s(\Omega)$-norm of functions whose $s$-th order derivatives are cartoon-like. Such functions appear as solutions of elliptic differential equations if the right hand side is a derivative of a cartoon-like function, as described in Subsection \ref{eq:regOfSols}. From now on, let $\Omega\subset \R^2$ be a bounded domain, whose boundary has finite length.

\begin{theorem}\label{thm:optApproxHs}
Let $s\in \N$, $\nu >0,$ $c\in \R^2$ and $\phi, \psi, \psitilde \in L^2(\R^2)$. For $|\mathbf{a}|\leq s,$ let 
\[
\left\{(D^{\mathbf{a}}\psi)_{j,k,m,\iota}:(j,k,m,\iota)\in \Lambda\right\}:= \mathcal{SH}\left(D^{\mathbf{a}}\phi, D^{\mathbf{a}}\psi, D^{\mathbf{a}}\psitilde,c\right)
\]
and let $(\theta_n^\mathbf{a}(\mathfrak{f}))_{n\in \N}$ be a non-increasing rearrangement of 
$(|\langle \mathfrak{f}, (D^{\mathbf{a}} \psi)_{j,k,m,\iota}\rangle|_{L^2(\Omega)})_{(j,k,m,\iota) \in \Lambda}$.
Further assume that for all $\mathfrak{f}\in \mathcal{E}^2(\nu, \Omega)$ we have that
\begin{align}\label{eq:assumptionApprox}
\theta_N^\mathbf{a}(\mathfrak{f}) \lesssim N^{-\frac{3}{2}}\log^{\frac{3}{2}} N, \text{ for all } N \in \N,
\end{align}
for all $|\mathbf{a}|\leq s$.
Furthermore, let $(\omega_{j,m,\upsilon})_{(j,m,\upsilon) \in \Theta}$ be a wavelet system on $\Omega$ such that for all $(j, m, \upsilon) \in \Theta$:
\begin{itemize}
    \item $|\suppp \omega_{j, m, \upsilon}| \lesssim 2^{-2j}$,
    \item $\|D^{\mathbf{a}} \omega_{j, m,\upsilon}\|_\infty \lesssim 2^{(|{\mathbf{a}}| + 1)j}$, 
    \item $\omega_{j, m, \upsilon}$ has two vanishing moments for all $\upsilon \neq 0$.
\end{itemize} 
On top of that, let $t\in \N$, $\tau >1/3$ and $(\varphi_n)_{n\in \N}:= \mathcal{HSW}^s_{t,\tau}(\Wcal;\phi, \psi, \psitilde, c)$. Then, for all $u$ such that $D^{\mathbf{a}} u \in \mathcal{E}^2(\nu, \Omega)$ for all $|{\mathbf{a}}| = s,$ we have
$$
\Xi_N(u) \lesssim N^{-\frac{3}{2}}\log^{\frac{3}{2}} N \text{ for all } N \in \N,
$$
where $(\Xi_m(u))_{m\in \N}$ is a non-increasing rearrangement of 
$ (|\langle u, 2^{-j_n s }\varphi_n \rangle_{H^s(\Omega)}|)_{n\in \N}$.

\end{theorem}
\begin{proof}
First of all, we have that with $(\Xi_m(u))_{m\in \N}$ being a non-increasing rearrangement of \\
$ {(|\langle u, 2^{-j_n s }\varphi_n \rangle_{H^s(\Omega)}|)_{n\in \N}}$ that
\begin{align} \label{eq:thesplitting}
\Xi_N(u) \lesssim \theta^{\mathrm{sh}}_{\lfloor N/2 \rfloor}(u) + \theta^{\mathrm{w}}_{\lfloor N/2 \rfloor} (u),
\end{align}
where $(\theta^{\mathrm{sh}}_n(u))_{n\in \N}$ is a non-increasing rearrangement of $(|\langle u, 2^{-js}\psi_{j,k,m,\iota} \rangle_{H^s(\Omega)}|)_{(j,k,m,\iota)\in \Lambda_0}$ and $(\theta^{\mathrm{w}}_n(u))_{n\in \N}$ is a non-increasing rearrangement of $(|\langle u, 2^{-js}\omega_{j,m,\upsilon} \rangle_{H^s(\Omega)}|)_{(j,m,\upsilon)\in \Theta_{t,\tau}}$.

We start by estimating $\theta^{\mathrm{sh}}_N(u)$ and we use that
\begin{align}\label{eq:SobolevNormDef1}
|\langle  u, 2^{-js}\psi_{j,k,m,\iota}  \rangle_{H^s(\Omega)}| \leq \sum_{0 \leq |{\mathbf{a}}|\leq s}|\langle  D^{\mathbf{a}} u, 2^{-js}D^{\mathbf{a}} \psi_{j,k,m,\iota} \rangle_{L^2(\Omega)}|.
\end{align}
Let us now assume without loss of generality that $\iota = 1,$ then there holds for $0\leq |\mathbf{a}|\leq s$ that 
\begin{align}
    | \langle D^{\mathbf{a}}u, 2^{-js} D^{{\mathbf{a}}} \psi_{j,k,m,1} \rangle_{L^2(\Omega)} | \leq & ~| \langle D^{\mathbf{a}}u, 2^{-j|{\mathbf{a}}|}D^{{\mathbf{a}}} \psi_{j,k, m, 1} \rangle_{L^2(\Omega)} | \nonumber\\
    ~ = & ~| \langle \widehat{D^{\mathbf{a}}u}, 2^{-j|{\mathbf{a}}|}(2\pi i\xi)^{\mathbf{a}} \widehat{\psi}_{j,k, m, 1} \rangle_{L^2(\R^2)} | \label{eq:plugInHere}\\
      ~ \leq & ~| \langle \widehat{D^{\mathbf{a}}u},  (2\pi i)^s 2^{-j a_2/2}( S_{-k}^T A_{-j} \xi)^{\mathbf{a}} \widehat{\psi}_{j,k, m, 1} \rangle_{L^2(\R^2)} |\nonumber \\&\quad + | \langle \widehat{D^{\mathbf{a}}u}, (2\pi i)^s (2^{-j a_2/2}( S_{-k}^T A_{-j}\xi)^{\mathbf{a}}  - 2^{-j|\mathbf{a}|} \xi^{\mathbf{a}}) \widehat{\psi}_{j,k, m, 1} \rangle_{L^2(\R^2)} |\nonumber \\
      ~ = & ~ \mathrm{I} + \mathrm{II}.\nonumber
\end{align}
Part $\mathrm{I}$ can be estimated by
\begin{align*}
    \mathrm{I} \leq | \langle \widehat{D^{\mathbf{a}}u},  \widehat{(D^{\mathbf{a}}\psi)}_{j,k, m, 1} \rangle_{L^2(\R^2)}|=|\langle D^{\mathbf{a}}u,  (D^{\mathbf{a}}\psi)_{j,k, m, 1} \rangle_{L^2(\Omega)}|.
\end{align*}
We continue with Part $\mathrm{II}$ and observe that 
$$
2^{-j a_2/2}(S_{-k}^T A_{-j}\xi)^{\mathbf{a}}  - 2^{-j|{\mathbf{a}}|}\xi^{\mathbf{a}} = \sum_{|\mathbf{b}| \leq |\mathbf{a}|, b_2 < a_2} c_{b_1,b_2} \xi^{\mathbf{b}},
$$
where $c_{b_1,b_2} \lesssim 2^{-j |\mathbf{a}|}$. Hence, to estimate $\mathrm{II}$ it suffices to estimate 
\begin{align}\label{eq:theEnd}
    &| \langle \widehat{D^{\mathbf{a}}u},(2\pi i)^s 2^{-j|{\mathbf{a}}|} \xi^\mathbf{b} \widehat{\psi}_{j,k, m, 1} \rangle_{L^2(\R^2)} | ,
\end{align}
where $b_2<{a}_2$. At this point we are again in the situation of \eqref{eq:plugInHere} and continue the estimate until we finally arrive at a point, where $a_2 = 0$. In this situation, we have 
\begin{align*}
   | \langle \widehat{D^{\mathbf{a}}u}, (2\pi i)^s 2^{-j|{\mathbf{a}}|} \xi_1^{a_1} \widehat{\psi}_{j,k, m, 1} \rangle_{L^2(\R^2)} | ~=& ~ | (2\pi i)^{s-a_1}\langle \widehat{D^{\mathbf{a}}u}, 2^{-ja_1} (2\pi i\xi_1)^{a_1} \widehat{\psi}_{j,k, m, 1} \rangle_{L^2(\R^2)} |\\
  ~ =& ~ | (2\pi i)^{s-a_1}\langle \widehat{D^{\mathbf{a}}u}, (2\pi i S_{-k}^T A_{-j} \xi)^{(a_1,0)} \widehat{\psi}_{j,k, m, 1} \rangle_{L^2(\R^2)} |\\
  ~ = & ~ |(2\pi i)^{s-a_1} \langle \widehat{D^{\mathbf{a}}u},\widehat{(D^{\mathbf{(a_1,0)}}\psi)}_{j,k, m, 1} \rangle_{L^2(\R^2)} |\\
  ~ = & ~ |(2\pi i)^{s-a_1} \langle D^{\mathbf{a}}u, (D^{\mathbf{(a_1,0)}}\psi)_{j,k, m, 1} \rangle_{L^2(\Omega)} |.
\end{align*} 
From the considerations above we can estimate for some $C>0$
\begin{align}\label{eq:firstHalf}
      \theta^{\mathrm{sh}}_N(u) \lesssim \sum_{|{\mathbf{a}}| \leq s}\sum_{|{\mathbf{b}}| \leq s} \theta^{\mathrm{sh}, \mathbf{b}}_{\lfloor N/C \rfloor}(D^{\mathbf{a}} u),
\end{align}
 where $(\theta^{\mathrm{sh}, \mathbf{b}}_N(g))_{N\in \N}$ denotes the non-increasing rearrangement of $(|\langle  g , 2^{-js}(D^{\mathbf{b}} \psi)_{j,k,m,\iota} \rangle|_{L^2(\Omega)})$. By \eqref{eq:assumptionApprox} we can estimate \eqref{eq:firstHalf} by $N^{-\frac{3}{2}}\log^{\frac{3}{2}} N$ for all $N \in \N$. This yields the estimate of $\theta^{\mathrm{sh}}_N(u)$. 
 We continue by estimating $\theta^{\mathrm{w}}_N(u)$. Assume first that $\suppp \omega_{j, m, \upsilon}$ intersects the singularity curve of $D^{\mathbf{a}}u,$ which shall be denoted by $\gamma^{\mathbf{a}}$ for $0\leq |\mathbf{a}|\leq s.$  Recall that $|\suppp \omega_{j, m, \upsilon}| \lesssim 2^{-2j}$ and $\|D^{\mathbf{a}} \omega_{j, m, \upsilon}\|_\infty \leq 2^{(|{\mathbf{a}}| + 1)j}$. Then, we have 
\begin{align*}
    |\langle  D^{\mathbf{a}}u, 2^{-js}D^{\mathbf{a}} \omega_{j, m, \upsilon} \rangle_{L^2(\Omega)} | &\leq 2^{-js}\|D^{\mathbf{a}}u\|_\infty \|D^{\mathbf{a}} \omega_{j, m, \upsilon}\|_\infty |\suppp \omega_{j, m, \upsilon}| \\
    &\lesssim 2^{-j(s + 2)}2^{(|{\mathbf{a}}| + 1)j} \leq  2^{-j(s + 2)}2^{j(s + 1)} = 2^{-j}.
\end{align*}
From the width of $\Gamma_{\tau(j-t)}$ we know that we have, up to some multiplicative constant, only $2^{(1-\tau)j}$ wavelet elements on scale $j$ intersecting $\gamma^{\mathbf{a}}$. Consequently, since $\tau>1/3,$ we get that \[(|\langle D^{\mathbf{a}}u,2^{-js}D^{\mathbf{a}}\omega_{j, m, \upsilon}\rangle_{L^2(\Omega)}|)_{(j,m,\upsilon) \in \Theta_{t,\tau}, (\suppp \omega_{j, m, \upsilon} \cap \gamma^\mathbf{a}) \neq \emptyset} \in \ell^{\frac{2}{3}}.\]
Lastly, we want to address the wavelets associated with the smooth part of $D^{\mathbf{a}}u$. We assumed that $D^{\mathbf{a}} \omega_{j, m, \upsilon}$ has two vanishing moments for all $j, m$ and $\upsilon \neq 0$ and hence by a taking a Taylor approximation of $D^{\mathbf{a}}u$ on $\suppp \omega_{j, m, \upsilon}$ we obtain
\begin{align*}
    |\langle  D^{\mathbf{a}}u, 2^{-js}D^{\mathbf{a}} \omega_{j, m, \upsilon} \rangle_{L^2(\Omega)} | \lesssim 2^{-js}\|D^{\mathbf{a}} \omega_{j, m, \upsilon} \|_\infty 2^{-2j} 2^{-2j}\lesssim 2^{-3j}.
\end{align*}

Since there are $2^{(2-\tau)j}$ wavelets on every scale we obtain that \[(|\langle D^{\mathbf{a}}u,2^{-js}D^{\mathbf{a}}\omega_{j, m, \upsilon}\rangle_{L^2(\Omega)}|)_{(j,m,\upsilon) \in \Theta_{t,\tau}, (\suppp \omega_{j, m, \upsilon} \cap \gamma^\mathbf{a}) = \emptyset} \in \ell^{\frac{2}{3}}.\]
Combining both estimates for the wavelets, we obtain that \[(|\langle D^{\mathbf{a}}u,2^{-js}D^{\mathbf{a}}\omega_{j, m, \upsilon}\rangle_{L^2(\Omega)}|)_{(j,m,\upsilon) \in \Theta_{t,\tau}}\in \ell^{\frac{2}{3}}.\] This yields that $\theta^{\mathrm{w}}_N(u) \lesssim N^{-\frac{3}{2}}$. Invoking \eqref{eq:thesplitting} concludes the proof.
\end{proof}

\begin{remark}
\begin{itemize}
\item The assumptions of Theorem \ref{thm:optApproxHs} are easily satisfied. As long as $((D^\mathbf{a}\psi)_{j,k,m,\iota})_{(j,k,m,\iota)\in\Lambda}$ is a shearlet system meeting the criteria for optimally sparse approximations of cartoon-like functions the assumptions are fulfilled. The necessary theoretical machinery has been established in \cite{KLcmptShearSparse2011}.
\item Theorem \ref{thm:optApproxHs} also establishes the coefficient decay of a shearlet frame for $H^s(\R^2)$. To observe this one can simply choose $\Omega\supset (0,1)^2$ sufficiently large so that no wavelet element intersects $(0,1)^2$. Then all scalar products of wavelets and $u$ are zero, demonstrating the result. 
\end{itemize}
\end{remark}

Theorem \ref{thm:optApproxHs} implies the following corollary, which describes the best $N$-term approximation rate of the reweighted hybrid shearlet-wavelet system, if it constitutes a frame.

\begin{corollary}\label{cor:optApproxHs}
Under the assumptions of Theorem \ref{thm:optApproxHs} and the additional assumption that $(2^{-j_n s} \varphi_n)_{n\in \N}$ is a frame for $H^s(\Omega)$ we have that
\begin{align*}
    \|u - u_N\|_{H^s(\Omega)}^2 \lesssim N^{-2} \log^{3}(N),
\end{align*}
where
$$
u_N := \sum_{n \in E_N} \langle u, 2^{-j_n s} \varphi_n \rangle_{H^s(\Omega)} (2^{-j_n s}\varphi_n)^{\mathrm{dual}}
$$
and $E_N$ contains the indices corresponding to the $N$ largest coefficients in modulus of $\langle u, 2^{-j_n s}\varphi_n \rangle_{H^s(\Omega)}$.
\end{corollary}

\begin{proof}
From the frame property we obtain that the synthesis operator of the dual frame is bounded and hence 
\begin{align*}
    \|u - u_N\|_{H^s(\Omega)}^2 \lesssim \|(\langle u, 2^{-j_n s} \varphi_n \rangle_{H^s(\Omega)})_{n\not \in E_N} \|_{\ell^2}^2.
\end{align*}
From Theorem \ref{thm:optApproxHs} we conclude that $\|(\langle u, 2^{-j_n s}\varphi_n \rangle_{H^s(\Omega)})_{n\not \in E_N} \|_{\ell^2}^2\lesssim N^{-2} \log^{3}(N)$.
\end{proof}

We will observe in Subsection \ref{sec:comp2Wave} that the approximation rate of the hybrid shearlet-wavelet system is much faster than that provided by a pure wavelet system. Moreover, we can also discuss whether the approximation rate of 
Corollary \ref{cor:optApproxHs} is optimal. If $s=2$, the discussion in Subsection \ref{eq:regOfSols} indicates that every cartoon-like function $g$ can be expressed as $g = \Delta u$ for a function $u$ such that $D^\mathbf{b}u \in \mathcal{E}^2(\nu, \Omega)$ for all multi-indices $|\mathbf{b}|\leq s$. 
While the arguments of Subsection \ref{eq:regOfSols} are technically only proved for piecewise $C^\infty$ functions, they are expected to also hold for piecewise $C^2$ functions.

Of course this is not a formal proof, since the argument of Subsection \ref{eq:regOfSols} only holds for piecewise smooth functions $g$.

However, if we assume the above statement to be true, then Corollary \ref{cor:optApproxHs} implies that the new hybrid shearlet-wavelet system yields almost the optimal approximation rate, due to the following argument. Assume we have a dictionary $\Phi$. Further assume that for all $u$ such that $D^\mathbf{b}u \in \mathcal{E}^2(\nu, \Omega)$ for all $|\mathbf{b}|\leq 2$ we have that $\Phi$ provides an $N$-term approximation rate of
$$
\|u - u_N\|_{H^2(\Omega)}^2 \lesssim N^{-2-\epsilon},
$$
for some $\epsilon>0$, where $u_N$ is the best $N$-term approximation of $u$ with respect to $\Phi$. Clearly this also implies
$$
\left\|D_{x_{1}}^{2} u - u_N^{x_1}\right\|_{L^2(\Omega)}^2 \lesssim N^{-2-\epsilon},~
\left\|D_{x_{2}}^{2} u - u_N^{x_2}\right\|_{L^2(\Omega)}^2 \lesssim N^{-2-\epsilon},
$$
where $u_N^{x_\iota}$ is the best $N$-term approximation of $D_{x_{\iota}}^{2} u$ with respect to the system $\Phi_{\iota}: = \{D_{x_{\iota}}^{2} \phi: \phi \in \Phi\}.$ Consequently $\Phi_{1} \cup \Phi_{2}$ yields a best $N$-term approximation rate of $D_{x_{1}}^{2} u + D_{x_{2}}^{2} u = \Delta u$ which is faster than $O(N^{-1})$. Such a result would be in conflict with the well-known optimality result of \cite{DCartoonLikeImages2001} which states that the optimal achievable best $N$-term approximation rate for the class of cartoon-like functions is given by $\sigma_{N,\Phi}(g) = O(N^{-1}),$ for cartoon-like functions $g$,
provided that only polynomial depth search is used to compute the approximation. Hence the approximation rate of Corollary \ref{cor:optApproxHs} is very likely to be optimal.

This argument can not be extended easily to $s \neq 2$ since it is not necessarily true that a given cartoon-like function $g$ is the $s$-th derivative of a function $u$ such that $D^\mathbf{b}u \in \mathcal{E}^2(\nu, \Omega)$ for all multi-indices $|\mathbf{b}|\leq s$. 

Now we turn our attention to a further important case where the $(s+l)$-th derivatives of $u$ are cartoon-like. Again, such functions appear frequently as solutions of elliptic differential equations with cartoon-like right-hand side (see Subsection \ref{eq:regOfSols}).

\begin{theorem}\label{thm:optApproxHsDerivative}
Let $s\in \N$, $0<l \in \N$, $\nu >0$ and $\phi, \psi, \psitilde \in L^2(\R^2),$ and $c\in \R^2$. For $|\mathbf{a}|\leq s,$ let 
\[
    \left\{(D^{\mathbf{a}}\psi)_{j,k,m,\iota}:(j,k,m,\iota)\in \Lambda\right\}:= \mathcal{SH}\left(D^{\mathbf{a}}\phi, D^{\mathbf{a}}\psi,D^{\mathbf{a}}\psitilde,c\right).
\] 
Further assume that for all $\mathfrak{f}\in L^2(\Omega)$ such that $D^\mathbf{b}\mathfrak{f} \in \mathcal{E}^2(\nu, \Omega)$ for all $|\mathbf{b}| \leq l$ we have that 
\begin{align} \label{eq:Assumptionsplusl}
\|(\langle \mathfrak{f}, (D^{\mathbf{a}} \psi)_{j,k,m,\iota}\rangle_{L^2(\Omega)})_{(j,k,m,\iota) \in \Lambda}\|_{\ell^{\frac{2}{l+3}}} < \infty,
\end{align}
for all $|\mathbf{a}|\leq s$. 
Furthermore, let $(\omega_{j,m,\upsilon})_{(j,m,\upsilon) \in \Theta}$ be a wavelet system on $\Omega$ such that for all $(j, m, \upsilon) \in \Theta$:
\begin{itemize}
    \item $|\suppp \omega_{j, m, \upsilon}| \lesssim 2^{-2j}$,
    \item $\|D^{\mathbf{a}} \omega_{j, m, \upsilon}\|_\infty \lesssim 2^{(|{\mathbf{a}}| + 1)j}$, 
    \item $\omega_{j, m, \upsilon}$ has $l+2$ vanishing moments for all $\upsilon \neq 0$.
\end{itemize} 
If $\tau >1/3$, then 
$$
\|(\langle u, 2^{-j_n s }\varphi_n \rangle_{H^s(\Omega)})_{n\in \N}\|_{\ell^{\frac{2}{l+3}}} < \infty.
$$
for all $u$ such that $D^{\mathbf{a}} u \in \mathcal{E}^2(\nu, \Omega)$ for all $|{\mathbf{a}}| = s+l$.
\end{theorem}
\begin{proof}
First of all, for $p = \frac{2}{l+3}$ we can certainly make the decomposition
\begin{align*}
    \|(\langle u, 2^{-j_n s }\varphi_n \rangle_{H^s(\Omega)})_{n\in \N}\|_{\ell^{p}}^{p} =  &\|(\langle u, 2^{-js}\psi_{j,k,m,\iota}\rangle_{H^s(\Omega)})_{(j,k,m,\iota) \in \Lambda_0}\|_{\ell^{p}}^{p}  \\ &+\|(\langle u, 2^{-js}\omega_{j,m,\upsilon}\rangle_{H^s(\Omega)})_{(j,m,\upsilon) \in \Theta_{t,\tau}}\|_{\ell^{p}}^{p}
    = \ \mathrm{I} + \mathrm{II}.
\end{align*}
Thus, it sufficies to prove the finiteness of $\mathrm{I}$ and $\mathrm{II}$. We start with $\mathrm{I}$. Repeating the computations from \eqref{eq:SobolevNormDef1} to \eqref{eq:theEnd}, we obtain that 
\begin{align*}
    \mathrm{I} \lesssim \sum_{|{\mathbf{a}}| \leq s}\sum_{|{\mathbf{b}}| \leq s} \| (\langle D^{\mathbf{a}} u, (D^{\mathbf{b}}\psi)_{j,k,m,\iota}\rangle_{L^2(\Omega)})_{(j,k,m,\iota) \in \Lambda_0}\|_{\ell^{p}}^p < \infty,
\end{align*}
where the finiteness follows by the assumption on the regularity of $u$ and \eqref{eq:Assumptionsplusl}. 

We continue by estimating $\mathrm{II}$. Assume first that $\suppp \omega_{j, m, \upsilon}$ intersects the discontinuity of $D^{\mathbf{a}}u$. Recall that $|\suppp \omega_{j, m, \upsilon}| \lesssim 2^{-2j}$ and $\|D^{\mathbf{a}} \omega_{j, m, \upsilon}\|_\infty \lesssim 2^{(|{\mathbf{a}}| + 1)j}$ and that $\omega$ has more than $l$ vanishing moments. Thus, using a Taylor approximation of $D^{\mathbf{a}}u$ on $\suppp \omega_{j, m, \upsilon}$ we can estimate 
\begin{align*}
    |\langle  D^{\mathbf{a}}u, 2^{-js}D^{\mathbf{a}} \omega_{j, m, \upsilon} \rangle_{L^2(\Omega)} | &\lesssim 2^{-js} 2^{-lj} \|D^{\mathbf{a}} \omega_{j, m, \upsilon}\|_\infty |\suppp \omega_{j, m, \upsilon}| \lesssim  2^{-(l+1)j}.
\end{align*}
From the width of $\Gamma_{\tau(j-t)}$ we know that, up to some constant, we have only $2^{(1-\tau)j}$ wavelet elements intersecting the boundary curve of $D^{\mathbf{a}} u$, which we denote by $\gamma^{\mathbf{a}}$. 

Consequently, we get that $(|\langle D^{\mathbf{a}}u,2^{-js}D^{\mathbf{a}}\omega_{j, m, \upsilon}\rangle_{L^2(\Omega)}|)_{(j,m,\upsilon) \in \Theta_{t,\tau}, (\suppp \omega_{j, m, \upsilon} \cap \gamma^\mathbf{a}) \neq \emptyset} \in \ell^{p}$ since $\tau >1/3$.
We proceed by estimating the norm associated to the wavelets that do not intersect $\gamma^\mathbf{a}$. We assumed that $D^{\mathbf{a}} \omega_{j, m, \upsilon}$ has $l+2$ vanishing moments for all $j, m$ and $\upsilon \neq 0$ and hence
\begin{align*}
    |\langle  D^{\mathbf{a}}u, 2^{-js}D^{\mathbf{a}} \omega_{j, m, \upsilon} \rangle_{L^2(\Omega)} | \lesssim 2^{-js}\|D^{\mathbf{a}} \omega_{j, m, \upsilon} \|_\infty 2^{-(l+2)j} 2^{-2j}\lesssim 2^{-(l+3)j}.
\end{align*}
Since there are $2^{(2-\tau)j}$ wavelets on every scale we obtain that  \[(|\langle D^{\mathbf{a}}u,2^{-js}D^{\mathbf{a}}\omega_{j, m, \upsilon}\rangle_{L^2(\Omega)}|)_{(j,m,\upsilon) \in \Theta_{t,\tau}, (\suppp \omega_{j, m, \upsilon} \cap \gamma^\mathbf{a}) = \emptyset} \in \ell^{p}.\]
Combining both estimates for the wavelets we obtain that \[(|\langle u,2^{-js} \omega_{j, m, \upsilon}\rangle_{H^s(\Omega)}|)_{(j,m,\upsilon) \in \Theta_{t,\tau}}\in \ell^{p}.\] This finishes the estimate of $\mathrm{II}$ and thus the proof. 
\end{proof}

\begin{remark}
The assumptions of Theorem \ref{thm:optApproxHsDerivative} can be easily checked by considering the results in \cite{Pet2014discontderiv}, which yields the required approximation rates for functions with cartoon-like derivatives by shearlets.
\end{remark}

We can deduce the following corollary establishing best $N$-term approximation rates in the case that the reweighted hybrid shearlet-wavelet system yields a frame.

\begin{corollary}\label{cor:optApproxHsDerivative}
Under the assumptions of Theorem \ref{thm:optApproxHsDerivative} and the additional assumption that $(2^{-j_n s} \varphi_n)_{n\in \N}$ is a frame for $H^s(\Omega)$ we have that
\begin{align*}
    \|u - u_N\|_{H^s(\Omega)}^2 \lesssim N^{-(l+2)},
\end{align*}
where
$$
u_N := \sum_{n \in E_N} \langle u, 2^{-j_n s} \varphi_n \rangle_{H^s(\Omega)} (2^{-j_n s}\varphi_n)^{\mathrm{dual}},
$$
and $E_N$ contains the indices corresponding to the $N$ largest coefficients $\langle u, 2^{-j_n s}\varphi_n \rangle_{H^s(\Omega)}$.
\end{corollary}
\begin{proof}
By Theorem \ref{thm:optApproxHsDerivative} we obtain that $(\langle u, 2^{-j_n s} \varphi_n \rangle_{H^s(\Omega)})_{n\in \N}\in \ell^{\frac{2}{l+3}}\hookrightarrow \ell^{\frac{2}{l+3}}_w.$ Invoking \eqref{BestN-Term-weak} yields the result.
\end{proof} 

\subsection{Comparison to wavelets}\label{sec:comp2Wave}

To assess the quality of the approximation rates achieved by hybrid shearlet-wavelet systems in Corollaries \ref{cor:optApproxHs} and \ref{cor:optApproxHsDerivative}, we should put them into perspective with approximation rates which are achievable by pure wavelet systems.

We will observe with the help of the following theorem that the approximation rate of Corollary \ref{cor:optApproxHs} cannot be achieved by pure wavelet systems.

\begin{theorem}[\cite{Coh2000}]
Assume that for $t \in [s, s']$ with $s'>s$ the spaces $B_{q,q}^t(\Omega)$, $t-s = d/q - d/p$ admit a wavelet characterization of the form
$$
\|u\|_{B_{p,p}^t(\Omega)} \sim \|(2^{tj} 2^{2(\frac{1}{2} - \frac{1}{p})j}  \|(c_\lambda)_{\lambda \in \Theta_j} \|_{\ell^p})_{j\geq 0} \|_{\ell^p}, 
$$
for all $u = \sum_{\lambda \in \Theta} c_\lambda \omega_\lambda$. Then, for $t \in ]s,s'[$, $t-s = 2/q - 2/p$, we have the norm equivalence 
\begin{align*}
    \|u\|_{B^t_{q,q}(\Omega)} \sim \|u\|_{B_{p,p}^s(\Omega)} + \| (2^{j(t-s)} \inf_{\mathfrak{g} \in \Sigma_j} \|u - \mathfrak{g}\|_{B_{p,p}^s(\Omega)})_{j\geq 0} \|_{\ell^q},
\end{align*}
where $\Sigma_j$ is the space of all $u = \sum_{n \in \Theta} c_n \omega_n$ with all but $2^{2j}$ coefficients $c_n$ equal to zero.
\end{theorem}

One can easily compute, that a function in $H^s(\Omega)$ whose $s$-th order partial derivatives are cartoon-like is not in $B^{s+1}_{q,q}(\Omega)$ for any $q$. Hence we obtain with $t = s+1$ and $p = 2$ and an accordingly chosen $q$ that 
$$
\| (2^{j} \inf_{\mathfrak{g} \in \Sigma_j} \|u - \mathfrak{g}\|_{H^s(\Omega)})_{j\geq 0} \|_{\ell^q},
$$
is not finite. Therefore it cannot hold that 
$$
\inf_{\mathfrak{g} \in \Sigma_j} \|u - \mathfrak{g}\|_{H^s(\Omega)} \lesssim 2^{-j-\epsilon},
$$
for some $\epsilon >0$. Consequently, since $\#\Sigma_j = 2^{2j}$, the best $N$-term approximation rate in $H^s(\Omega)$-norm does not decrease faster than $N^{-\frac{1}{2} -\epsilon}$ for any $\epsilon>0$. Invoking for some $l\in \N$ the same argument for a function $u \in H^{s+l}(\Omega)$ whose $(s+l)$-th derivative is cartoon-like yields that the approximation rate with respect to the $H^s(\Omega)$-norm will not decrease faster than $N^{-\frac{l+1}{2} -\epsilon}$ for any $\epsilon >0$. If one compares these results to the approximation rates by hybrid shearlet-wavelet systems of Corollaries \ref{cor:optApproxHs} and \ref{cor:optApproxHsDerivative} we observe that the approximation rates by reweighted hybrid shearlet-wavelet systems are always superior.

\section{Numerical examples}\label{sec:numExps}

We aim to demonstrate the future potential of shearlet systems in view of possible applications in discretization of PDEs. Towards this goal, we will provide some numerical examples for the solution of a Poisson equation with cartoon-like right-hand side using an adaptive discretization by the new shearlet system. 

\subsection{Adaptive solution of PDEs}
The adaptive algorithm is based on the works of \cite{cohen2001adaptive, Ste2003}, which establish a method to adaptively solve operator equations. In particular, let $L:\mathcal{H}\rightarrow \mathcal{H}'\cong \mathcal{H}$ be a linear, bounded and boundedly invertible operator that induces a symmetric and elliptic bilinear form. Moreover, let $f\in\mathcal{H}'$. We want to find the uniquely determined $u\in \mathcal{H}$ such that 
\begin{align} \label{eq:opereq}
Lu=f.  
\end{align} 
Using a frame for $\mathcal{H}$, the problem can now be transformed into a discrete one by setting $u=T^{*}_{\Phi}\textbf{u}.$ We define 
\begin{align}\label{eq:DefinitionOfL}
\textbf{L}:\ell^2(\Lambda)\rightarrow \ell^2(\Lambda),~\textbf{L}=T_{\Phi}LT^*_{\Phi}, \text{ and } \textbf{f}=T_{\Phi}f\in \ell^2(\Lambda).
\end{align}
It was shown in \cite{Ste2003} that, if $L$ fulfills the properties described above, solving \eqref{eq:opereq} is equivalent to solving the discretized system 
\begin{align} \label{eq:discsys}
\textbf{Lu}=\textbf{f}.
\end{align}

Additionally in \cite{cohen2001adaptive} an algorithm named \textbf{SOLVE} was developed that provides a solution to \eqref{eq:discsys} using adaptive refinements. Under certain assumptions on the mappings $\mathbf{L},\mathbf{P}_{\mathrm{ran}\mathbf{L}},$ a discretized solution $\mathbf{u}_N$ with $N$ non-zero entries can be calculated with $O(N)$ operations. The convergence rate of the $\mathbf{u}_N$ to the exact solution $\mathbf{u}$ of \eqref{eq:discsys} is given by the error of the best $N$-term approximation rate of $u$ with respect to $\Phi.$ Hence the algorithm's convergence rate is asymptotically optimal. 

\subsection{Setup}\label{sec:Setup}

Our theoretical findings show that the hybrid shearlet-wavelet system yields a frame that is very well suited to approximate functions with anisotropic characteristics such as jump singularities in their derivatives and, as we have seen, such anisotropic structures appear in the solutions of PDEs.

Because of this, we examine an implementation based on \textbf{SOLVE}, and analyze its convergence rate. The first step is to establish an implementation of the analysis and synthesis operators of a hybrid shearlet-wavelet system. 
Contrary to the aforementioned results, which are valid for wavelet systems $\mathcal{W}$ defined on general domains $\Omega,$ we choose $\Omega = (0,1)^2$ for our numerical example. This is because, our implementation is based on the standard wavelet library \texttt{WaveLab} \cite{Buckheit95wavelaband} 
and the shearlet library \texttt{ShearLab} \cite{shearLab}, which both operate under the assumption $\Omega =(0,1)^2$. For a given digitization of the domain $\Omega$ and a function $h$ on $\Omega$ modeled as an $n \times n$ pixel image, 
\texttt{WaveLab} provides an analysis operator $T_{\mathcal{W}}$, such that $T_{\mathcal{W}}(h)$ is a vector containing all boundary wavelet coefficients, computed with the $L^2$-scalar product, of a wavelet basis up to a certain scale. 
To turn this operator into an analysis operator based on the $H^1$-scalar product, we define 
\begin{align}\label{eq:TransformToH1}
\mathbf{T_{\mathcal{W}}}(h) := T_{\mathcal{W}}(h - \Delta_{\mathrm{discr}} h),
\end{align}
where $\Delta_{\mathrm{discr}}$ is Matlab's built-in discrete Laplacian. This definition is justified because we choose a wavelet system for which all involved wavelets vanish at the boundary of $\Omega$.

Similarly, \texttt{ShearLab} provides a built-in analysis operator $T_{\mathcal{S}}$. Using the same construction as \eqref{eq:TransformToH1} we construct an analysis operator $\mathbf{T}_{\mathcal{S}}$ with respect to the $H^1$-scalar product. Besides, the analysis operator of \texttt{ShearLab}, called $T_{\mathcal{S}}$, computes the coefficients associated to all possible translates, i.e., it effectively implements the analysis operator of the system
\begin{align}\label{eq:WhatWeOriginallyGet}
\left\{\psi_{j,k,(A_j m ), \iota} \colon 0 \leq j \leq J, \iota \in \{-1,0,1\}, |k|\leq |\iota| 2^{\lfloor j/2 \rfloor}, m \in c\Z \right\}, 
\end{align}
for a suitable $c > 0$ and a maximal scale $J$. However, comparing with Definition \ref{def:ShearletSystem}, we require an analysis operator for the system 
\begin{align}\label{eq:WhatWeWant}
\left\{\psi_{j,k, m , \iota} \colon 0 \leq j \leq J, \iota \in \{-1,0,1\}, |k|\leq |\iota| 2^{\lfloor j/2 \rfloor}, m \in c_1\Z, \suppp \psi_{j,k, m , \iota} \cap \partial \Omega = \emptyset \right\}
\end{align}
instead. To turn the analysis operator supplied by \texttt{ShearLab} into one more appropriate for our purposes, we introduce a mask $\mathbf{M_{\mathcal{S}}}$ which sets every coefficient associated with the system \eqref{eq:WhatWeOriginallyGet} that is not in \eqref{eq:WhatWeWant} to zero. Moreover, for $\tau, t>0$ we construct a mask  $\mathbf{M_{\mathcal{W}}}$ which sets every component of $\mathbf{T_{\mathcal{W}}}$ not associated to $\mathcal{W}_{t,\tau}$ to zero. Finally, to compute the coefficients with respect to the reweighted systems, we introduce two matrices $\mathbf{D}_{\mathcal{W}}$ and $\mathbf{D}_{\mathcal{S}}$ that multiply each wavelet or shearlet coefficient with $2^{-j}$ where $j$ is the associated scale. Ultimately, we now implement the analysis operator of the hybrid shearlet wavelet system by 
$$
\mathbf{T_{HSW}} := \binom{\mathbf{D}_{\mathcal{W}} \mathbf{M}_{\mathcal{W}} \mathbf{T}_{\mathcal{W}} }{\mathbf{D}_{\mathcal{S}} \mathbf{M}_{\mathcal{S}} \mathbf{T}_{\mathcal{S}}}. 
$$
In the case that $L$ denotes the Laplacian, the construction \eqref{eq:DefinitionOfL} then leads to the following discrete operator 
$$
\mathbf{L} := -\mathbf{T_{HSW}} \Delta_{\mathrm{discr}} \mathbf{T_{HSW}}^*,
$$
completing our implementation of the operator equation \eqref{eq:opereq}. 

Since the analysis operator of \texttt{ShearLab} computes the coefficients of the shearlet transform by using convolutions, it always computes all coefficients on a given scale. As a result, this transformation is inherently non-adaptive. To analyze the adaptive routine \textbf{SOLVE} we thus need to simulate the adaptivity by replacing the involved adaptive steps by threshholding. This is extensively documented in \cite[Section 4.3.1]{DissPP}.

\subsection{Experiment}

We solve a Poisson problem on a bounded domain $\Omega = (0,1)^2$ with right-hand side $f$
\begin{align} \label{eq:PoissonProblemExp}
-\Delta u &= f \text{ on }\Omega,\\
u &= 0 \text{ on }\partial \Omega,\nonumber
\end{align}
where $f= D_1g+D_2g$ and $g= \mathcal{X}_{B_{\frac{1}{6}}(0.5)},$
i.e. $g$ is a cartoon-like function and $f$ is the sum of the derivatives of $g.$ We have depicted both $g$ and $f$ in Figure \ref{fig:RHS}.

\begin{figure}[htb]
    \centering
    \includegraphics[width = 0.4 \textwidth]{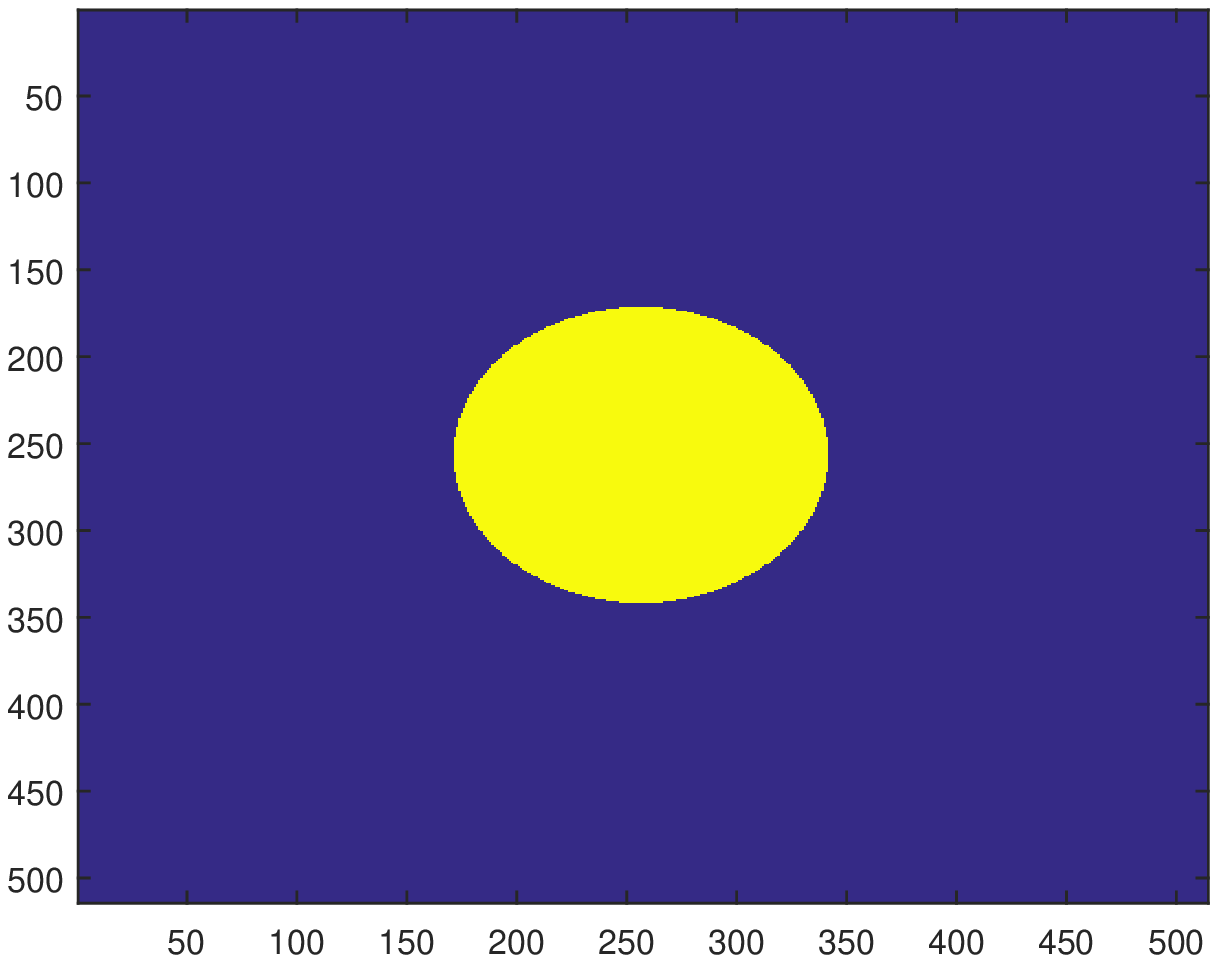}
    \includegraphics[width = 0.4 \textwidth]{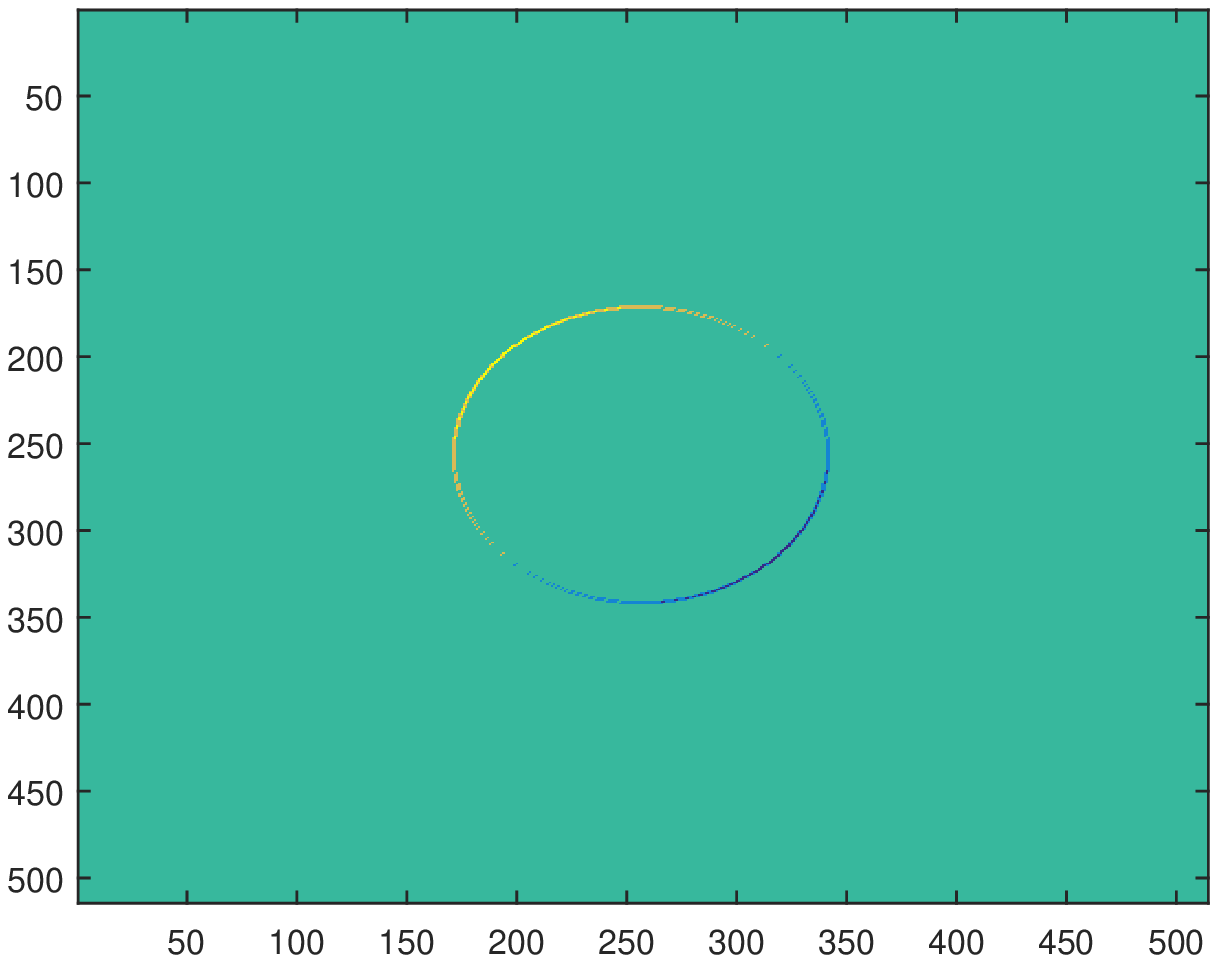}  \quad  
     \caption{\textbf{Left:} The cartoon-like function $g$. \textbf{Right:} The right-hand side $f=D_1g+D_2g$.}
    \label{fig:RHS}
\end{figure}

Moreover, the exact solution $u$ is depicted in Figure \ref{fig:Data} alongside with its derivatives. We can clearly see, that the derivatives of $u$ are cartoon-like functions. 
\begin{figure}[htb]
    \centering
    \includegraphics[width = \textwidth]{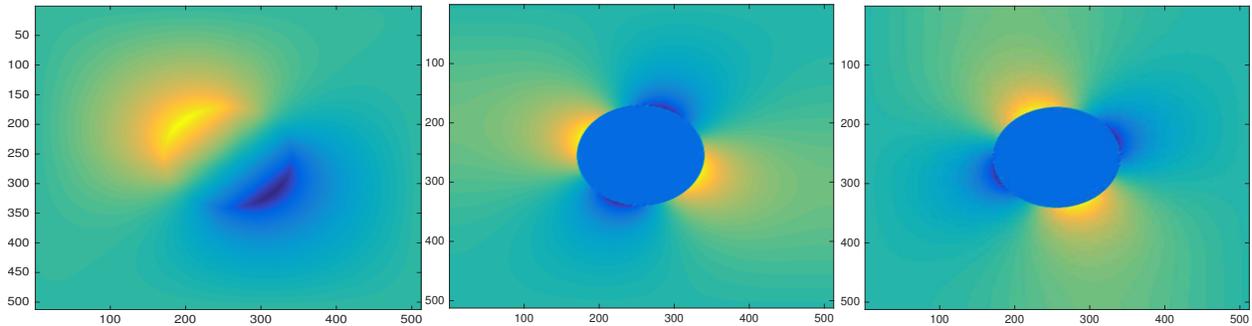}
    \caption{\textbf{Left:} Solution of the PDE \eqref{eq:PoissonProblemExp}. \textbf{Middle:} Cartoon-like derivative in $x_1$ direction, \textbf{Right:} Cartoon-like derivative in $x_2$ direction.}
    \label{fig:Data}
\end{figure}

\begin{figure}[htb]
    \centering
    \includegraphics[width = 0.4 \textwidth]{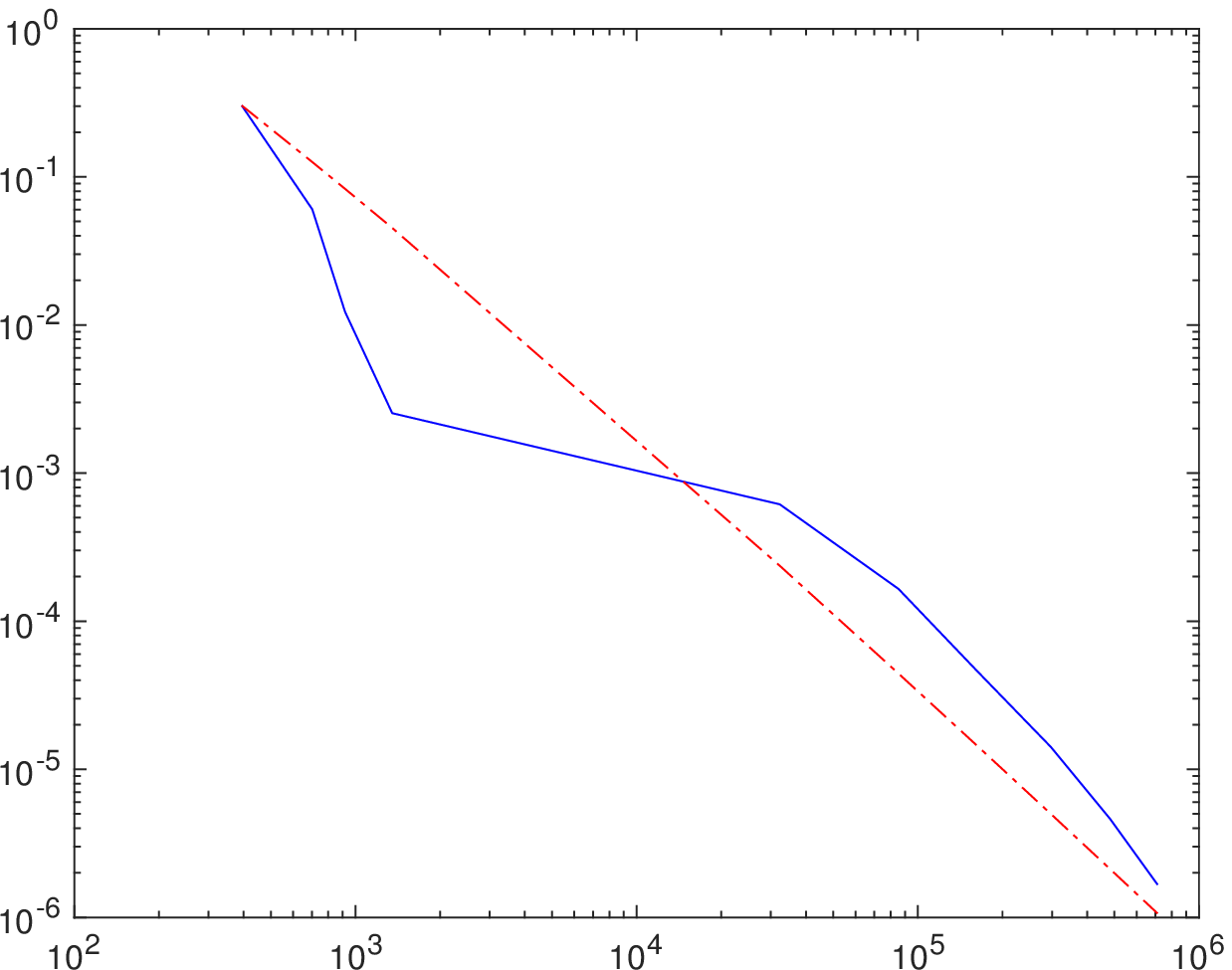}
    \includegraphics[width = 0.4 \textwidth]{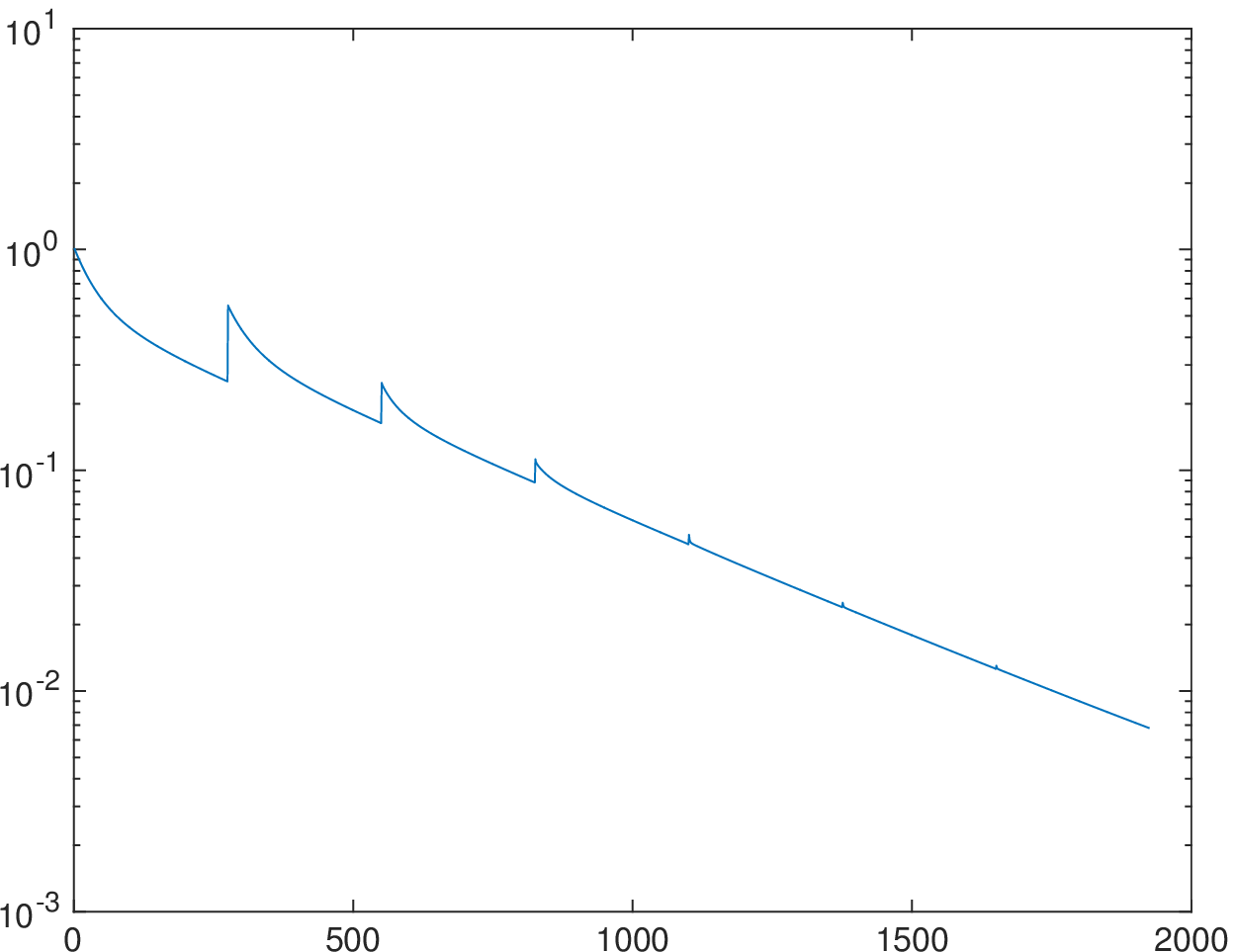}  \quad  
     \caption{\textbf{Left:} Squared error of approximation, blue actual approximation, red $N^{-1}$.  \textbf{Right:} Error in each iteration within \textbf{SOLVE} against number of iterations.}
    \label{fig:Rate}
\end{figure}

In Figure \ref{fig:Rate} we can observe an approximation rate of $O(N^{-1})$ after executing our version of \textbf{SOLVE}, defined in Subsection \ref{sec:Setup}. Moreover, since \textbf{SOLVE} yields an approximation rate of the order of the best $N$-term approximation rate, we conclude with the results of Subsection \ref{sec:comp2Wave} that the approximation rate of $O(N^{-1})$ is faster than that provided by wavelet frames, which can only achieve $O(N^{-1/2})$.  

\begin{figure}[htb]
    \centering
    \includegraphics[width = 1.1\textwidth]{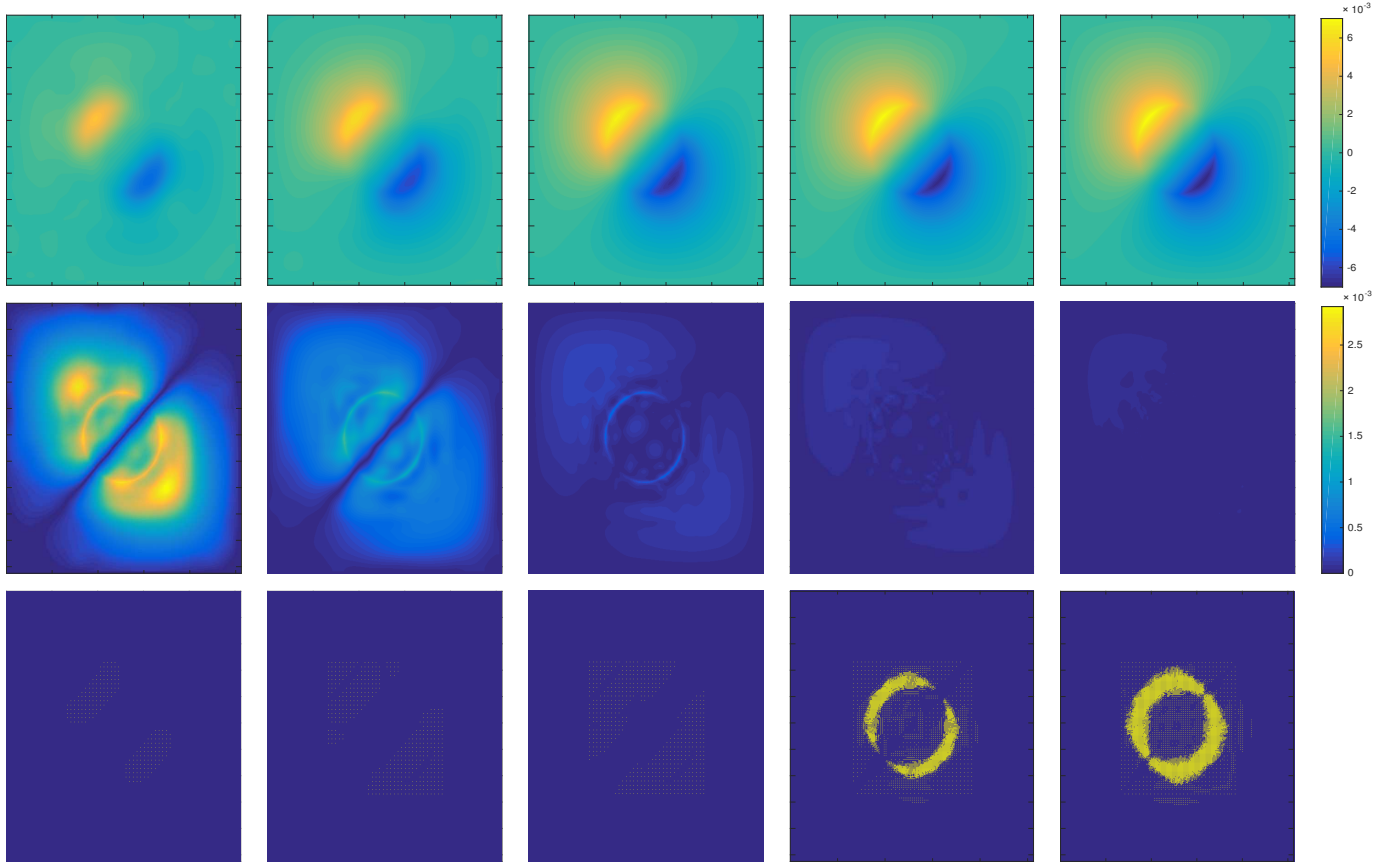}
     \caption{\textbf{Top:} Reconstructions $ \mathbf{T_{HSW}}^* v^{(i)}$ for $i = 1, \dots 5$; \textbf{Middle:} Errors of the reconstructions, i.e., $| \mathbf{T_{HSW}}^* v^{(i)} - u|$ for $i = 1, \dots 5$; \textbf{Bottom:} Active shearlet elements, i.e., the non-zero coefficients of the shearlet part of $v^{(i)}$.}
    \label{fig:Recs}
\end{figure}

On top of the study of the approximation rate, we can also analyze the approximation quality. Here we provide in Figure \ref{fig:Recs} the reconstructed functions, the error committed and the active shearlet elements. One can clearly observe, that the error vanishes uniformly, which demonstrates the good approximation quality of shearlets at positions of curvilinear singularities.

Additionally, we can observe, that the algorithm finds the elements that are most strongly associated with the jump singularity of the derivatives of $u$.

\subsection{Outlook}

The discretization of elliptic PDEs by the new system appears to be a very promising line of research. Towards a proper implementation of an adaptive scheme the following issues need to be analyzed and will be subject of future work.

\begin{itemize}
\item \emph{Mapping properties:} In order to obtain theoretical guarantees for the convergence rate of a hybrid shearlet-wavelet-based adaptive frame method, we certainly need to analyze the assumptions concerning the mapping properties of the discretized operator equation in advance. 
\item \emph{Approximation rates with respect to the primal frame:} The convergence rate of \textbf{SOLVE} depends on the $N$-term approximation rate provided by the underlying frame. Our theoretical results only provide approximation rates with respect to the dual of the hybrid shearlet-wavelet frame, which is standard in shearlet literature. Nonetheless, in order to guarantee the approximation rates of \textbf{SOLVE} a proper analysis of the $N$-term approximation rates with respect to the primal frame has to be conducted.
\item \emph{Implementation for a model problem:} As we mentioned before, our current implementation is based on thresholding procedures because it is based on the available shearlet code. Clearly, an implementation of a hybrid shearlet-wavelet-based solver that carries out all the adaptive routines properly, is desirable and will be developed in the future.
\end{itemize}

\section*{Acknowledgments}
P. Petersen and M. Raslan thank P. Grohs and G. Kutyniok for valuable discussions. P. Petersen and M. Raslan acknowledge support from the DFG Collaborative Research Center TRR 109 ``Discretization in Geometry and Dynamics'' and the Berlin Mathematical School.  P. Petersen is supported by a DFG Research Fellowship ``Shearlet-based energy functionals for anisotropic phase-field methods".

\bibliographystyle{plain}
\bibliography{Bib.bbl}

\end{document}